\documentclass[11pt, a4paper]{article}
\usepackage{graphicx} 
\usepackage{hyperref} 
\usepackage{authblk} 
\usepackage{amsmath,amssymb,amsthm}
\usepackage{mathrsfs}
\usepackage{dsfont}
\usepackage{xcolor}
\usepackage{enumitem}

\usepackage[T1]{fontenc}
\usepackage{lmodern}
\usepackage[babel]{microtype}
\usepackage[english]{babel}

\usepackage{geometry}
\numberwithin{equation}{section}
\numberwithin{figure}{section}

\newtheorem{theorem}{Theorem}
\newtheorem{lemma}[theorem]{Lemma}

\newtheorem{claim}[theorem]{Claim}
\newtheorem{proposition}[theorem]{Proposition}
\newtheorem{remark}[theorem]{Remark}
\newtheorem{definition}[theorem]{Definition}
\newtheorem{conjecture}[theorem]{Conjecture}
\newtheorem{problem}[theorem]{Problem}

\newcommand{\eps}{\varepsilon}
\newcommand{\Prob}{\operatorname{Pr}}
\newcommand{\Exp}{\mathbb{E}}
\newcommand{\abs}[1]{\left\lvert #1 \right\rvert}
\newcommand{\Neighs}{N}
\newcommand{\NeighsKOut}[4]{\Neighs_{#2, #3}^{#1}\!\left(#4\right)}
\newcommand{\NeighsOutL}[3]{\NeighsKOut{L+}{#1}{#2}{#3}}
\newcommand{\NeighsOutR}[3]{\NeighsKOut{R+}{#1}{#2}{#3}}

\newcommand{\Degr}{d}
\newcommand{\DegrKOut}[4]{\Degr_{#2, #3}^{#1}\!\left(#4\right)}
\newcommand{\DegrOutL}[3]{\DegrKOut{L+}{#1}{#2}{#3}}
\newcommand{\DegrOutR}[3]{\DegrKOut{R+}{#1}{#2}{#3}}

\title{Rainbow connectivity Maker-Breaker game}

\author[1]{Juri Barkey}
\author[1]{Bruno Borchardt}
\author[1]{Dennis Clemens}
\author[2]{Milica Maksimovi\'c}
\author[2]{Mirjana Mikala\v{c}ki}
\author[2]{Milo\v{s} Stojakovi\'c}

\affil[1]{Hamburg University of Technology, Institute of Mathematics, Am Schwarzenberg-Campus 3, 21073 Hamburg, Germany}
\affil[2]{Department of Mathematics and Informatics, Faculty of Sciences, University of Novi Sad, Serbia}

\date{}

\begin{document}

\maketitle

\begin{abstract}
    We study biased Maker-Breaker games on a graph system $\{G_1,\ldots,G_s\}$, in which Maker's goal is to claim certain rainbow structures, i.e., specified subgraphs consisting of at most one edge from each graph $G_i$. We consider the rainbow-connectivity game,
    in which Maker wants to claim a rainbow path between every pair of vertices. 
    
    We analyse this game in detail, essentially determining the threshold bias when played on the system of complete graphs, and observing that whether the random graph intuition holds depends on the size of $s$. The key ingredient of our result is the analysis of a Maker's strategy that combines several randomized strategies with an appropriately designed balancing game. As a byproduct, we find the order of the threshold bias for the Maker-Breaker diameter game, and disprove a conjecture by Balogh, Martin and Pluh\'ar. 
    
    Another natural and general way to analyse Maker-Breaker games that are played on a colored board is to require Maker to occupy a rainbow winning set of a given positional game. In the case of the connectivity game, Maker's goal is to claim a rainbow spanning tree. For this game played on the system of complete graphs, we establish matching upper and lower bounds on the threshold bias, up to constant factors.
\end{abstract}

\section{Introduction}

\subsection{Rainbow Combinatorics}

A recent trend in combinatorics is the study of rainbow 
variants of central results in extremal and random graph theory, see, e.g.,
the survey~\cite{sun2024transversal}. One of the earliest and most famous problems in this direction is the well-known 
Ryser-Brualdi-Stein Conjecture~\cite{brualdi1991combinatorial,ryser1967neuere,stein1975transversals} 
on the existence of long transversals in Latin squares, which can be reformulated
as a question on rainbow matchings, and which was recently resolved for every large enough
Latin square of even size in the breakthrough paper of Montgomery~\cite{montgomery2023proof}. Partially motivated by this problem, one popular type of rainbow questions
in the literature considers the search for transversals and rainbow structures in systems of graphs.
The central definition for these problems is as follows.

\begin{definition}
Let $m\in \mathbb{N}$ and let 
$\mathcal{G} = \{G_1,G_2,\ldots,G_m\}$
be a system of graphs
on the same vertex set $V$.
Then a graph $H$ on $V$ with $m$ edges 
is called transversal in $G$ if it contains
exactly one edge of each of the graphs $G_i$,
that is, if there exists a bijection 
$\phi:  E(H) \rightarrow [m]$ such 
that $e$ is contained in $G_{\phi(e)}$ 
for every edge $e\in E(H)$.
Similarly, if $H$ has at most $m$ edges, 
then it is called rainbow in $G$ if it contains
at most one edge of each of the graphs $G_i$.
\end{definition}

Note that one may identify each of the graphs $G_i$
with a distinct color $c_i$. 
Then the union $G=(V,E)$ with $E=\bigcup_{i\in [m]} E(G_i)$
can be considered as a colored multigraph,
and $H\subseteq G$ being rainbow simply means that the edges of $H$ have distinct colors. 
Recent results in this area cover the cases when 
$H$
is a triangle~\cite{aharoni2020rainbow},
a perfect matching or a Hamilton cycle~\cite{anastos2023robust,bradshaw2022one,cheng2025transversal,
ferber2022dirac,joos2020rainbow},
or a more complex structure~\cite{bradshaw2021transversals,bradshaw2021rainbow,
chakraborti2023bandwidth,cheng2021rainbow,gupta2023general,heath2025universality}.
We quickly mention two highlights from this list of papers.  Let $\mathcal{G}=\{G_1,\ldots,G_n\}$ be a system of Dirac graphs, i.e., each graph has minimum degree at least $\frac{n}{2}$. Then Bradshaw et al.~\cite{bradshaw2022one}
prove that $\mathcal{G}$ has at least
$\left( \left(\frac{1}{64}-o(1) \right) n \right)!$ rainbow Hamilton cycles. Moreover, Anastos et al.~\cite{anastos2023robust}
show that there is a constant $c>0$ with the following property: If $G$ is a Dirac graph on $n$ vertices and if additionally $\{H_1,\ldots,H_n\}$ is a system of 
random graphs distributed independently like $G_{n,p}$,
with $p\geq \frac{c\log(n)}{n^2}$, then the system
$\{G\cap H_1,\ldots,G\cap H_n\}$ has a rainbow Hamilton cycle a.a.s.~(asymptotically almost surely). 

\medskip

In the following, we study biased Maker-Breaker games on graph systems, where we put a focus on two different notions of connectivity in the rainbow setting.

\subsection{Maker-Breaker games}

Given a hypergraph $\mathcal{H}=(X,\mathcal{F})$
and an integer $b\in \mathbb{N}$,
the $(1:b)$ Maker-Breaker game on $\mathcal{H}$
is played as follows. Maker and Breaker 
alternatingly claim 
unclaimed elements of the given \textit{board} $X$. In every round, Maker claims 1 element 
and then Breaker claims $b$ elements, until all the elements are claimed.
Maker wins if at the end of the game she
claimed all elements of a \textit{winning set} $F\in\mathcal{F}$, and Breaker wins otherwise, i.e., if he claimed at least one element in every winning set.

\smallskip

Maker-Breaker games have a long history of research,
starting from the seminal papers of Hales and Jewett~\cite{hales1963regularity} and Chv\'atal and Erd\H{o}s~\cite{chvatal1978biased}, with the biggest research output done
in the last roughly twenty years. For an overview on Maker-Breaker games and other positional games, we refer to the survey~\cite{krivelevich2014positional}
and the book~\cite{hefetz2014positional}.

\smallskip

An important fact about biased Maker-Breaker games, which was already observed by Chv\'atal and Erd\H{o}s~\cite{chvatal1978biased}, is that these game are bias monotone. That is, for any game $\mathcal{H}=(X,\mathcal{F})$, if Breaker wins the $(1:b)$ game, then he also wins the $(1:b+1)$ game. In particular,
as long as $\min\{|A|:A\in \mathcal{F}\}\geq 2$, there is a unique integer $b_{\mathcal{H}}$ such that Breaker wins the $(1:b)$ game on $\mathcal{H}$ if and only if $b\geq b_{\mathcal{H}}$.
This integer, referred to as the \textit{threshold bias} of the game $\mathcal{H}$, has become a central object of study in the theory on Maker-Breaker games, see, e.g.,~\cite{allen2017making, balogh2009diameter,bednarska2000biased,
ferber2015generating,gebauer2009asymptotic,hefetz2008planarity,
krivelevich2011critical,liebenau2022threshold}. 

\smallskip

As a prominent and well-studied example, let us consider the \textit{connectivity game} on the complete graph $K_n$. This game is played on the hypergraph $\mathcal{C}_n = (X,\mathcal{F})$, where $X$ is the edge-set of the complete graph, $X=E(K_n)$,
and the winning sets $\mathcal{F}$ are the edge sets 
of all spanning trees in $K_n$. Chv\'atal and Erd\H{o}s~\cite{chvatal1978biased} proved that $b_{\mathcal{C}_n}$ is of the order $\frac{n}{\log(n)}$, leaving it as an open problem to determine the precise constant of the leading term. This was resolved 30 years later by Gebauer and Szab\'o~\cite{gebauer2009asymptotic}, improving on an intermediate result by Beck~\cite{beck1982remarks} to show that $b_{\mathcal{C}_n} = (1+o(1))\frac{n}{\log(n)}$. 

Note that this result is surprising in the following way.
If Maker and Breaker would play purely at random,
Maker's final graph would be a random graph on
$\lceil \frac{1}{b+1}\binom{n}{2} \rceil$ edges.
Such a random graph is known to be connected a.a.s.~if $b\geq (1+\varepsilon)\frac{n}{\log(n)}$ and it is disconnected  a.a.s.~if $b\leq (1-\varepsilon)\frac{n}{\log(n)}$
(see, e.g.,~\cite{alon2016probabilistic,frieze2015introduction}). 
Hence, a game between two random players is ``quite likely'' to have the same winner as the game between two players playing optimally.

This exciting connection between Maker–Breaker games and random graphs is now widely recognized in the literature as the \textit{random graph intuition}, or the \textit{Erd\H{o}s paradigm}. The results by Chv\'atal and Erd\H{o}s and Gebauer and Szab\'o have pushed further the research on threshold biases and their connection to the random graph theory.
There is a number of other positional games that also behave as predicted by the random graph intuition -- at least up to a constant factor -- see, e.g.,~\cite{hefetz2008planarity,hefetz2014positional,krivelevich2011critical,nenadov2023probabilistic, nenadov2016threshold, stojakovic2005positional}.

\subsection{Maker-Breaker games on graph systems}

The aim of this paper is to determine threshold biases and analyse the random graph intuition for two different Maker-Breaker games on graph systems.
We start with the concept of rainbow-connectivity, which was introduced in~\cite{chartrand2008rainbow}.
Given an integer $s\in\mathbb{N}$ and a graph system $\mathcal{G}=\{G_1,\ldots,G_s\}$
on a vertex set $V$, we say that $\mathcal{G}$ is rainbow-connected
if every two vertices of $V$ are connected by a rainbow path in $\mathcal{G}$.
Accordingly, we define the following \textit{rainbow-connectivity game} $\mathcal{C}_{s,n}$.

\begin{definition}[Rainbow-connectivity game $\mathcal{C}_{s,n}$] \label{def:rainbow.conn.game}
Let $1\leq s\leq n$ be integers.
The $(1:b)$ game $\mathcal{C}_{s,n}$ is played as follows.
Given copies $G_1,\ldots,G_s$ of the complete graph
$K_n$ on vertex set $V=[n]$, Maker and Breaker alternately claim edges of 
the board $X=\bigcup_{i\in [s]}E(G_i)$ according to the given bias. Maker wins iff she manages to claim
edges in such a way that by the end of the game,
every two vertices of $V$ are connected by a rainbow path.
That is, she wins iff she claims spanning 
subgraphs $H_i\subseteq G_i$, for every $i\in [s]$, such that
the system $\{H_1,\ldots,H_s\}$ is rainbow-connected.
\end{definition}

\smallskip

Given a constant $s\in \mathbb{N}$ and 
a system $\mathcal{G}=\{G_1,\ldots,G_s\}$,
where each $G_i$ is distributed independently
like $G_{n,p}$, it is not hard to show that
the threshold value for $p$ such that $\mathcal{G}$
is rainbow-connected is of size
$n^{-\frac{s-1}{s}+o(1)}$. Indeed, a standard first moment argument shows that for 
$p \leq n^{-\frac{s-1}{s}-\varepsilon}$ the system 
$\mathcal{G}$ is not rainbow-connected a.a.s.,
while an application of Janson's inequality gives that for 
$p \geq n^{-\frac{s-1}{s}+\varepsilon}$ it is rainbow-connected a.a.s. Hence, having in mind the standard relations between uniform random graphs and binomial random graphs, the random graph intuition \emph{suggests} that the threshold bias of the game 
$\mathcal{C}_{s,n}$ should be of the order $n^{\frac{s-1}{s}+o(1)}$. However, our first two results show that the random graph intuition fails in this case.

\begin{proposition}\label{prop:rainbow-2-conn}
Let $n\geq 2$ be an integer. The threshold bias for the rainbow-connectivity game $\mathcal{C}_{2,n}$ 
on $s=2$ copies of $K_n$ is 
$b_{\mathcal{C}_{2,n}}=2.$
\end{proposition}

\begin{theorem}\label{thm:rainbow-conn.s.constant}
Let $s\geq 3$. 
The threshold bias for the rainbow-connectivity game $\mathcal{C}_{s,n}$ 
on $s$ copies of $K_n$ is of the order
$n^{1-1/\lceil s/2 \rceil}$. More precisely,
there exist constants $C,\beta,\gamma>0$, depending only on $s$, such that the following holds for every large enough $n$. 
\begin{enumerate}
    \item[(a)] If $b\geq Cn^{1-1/\lceil s/2 \rceil} $, then Breaker wins the game.
    \item[(b)] If $b\leq \gamma n^{1-1/\lceil s/2 \rceil}$ then Maker wins the game. Moreover, she has a strategy with which she obtains $\beta n^{s-1} b^{-s}$ rainbow paths of length $s$ between every pair of vertices and for every ordering of the $s$ colors. 
\end{enumerate}
\end{theorem}

We note that 
the bound $\beta n^{s-1} b^{-s}$ on the number of rainbow paths is optimal up to the constant factor $\beta$, as can by shown by a potential function argument, see our concluding remarks in Section~\ref{sec:concluding} for details. Moreover, our proof of Theorem~\ref{thm:rainbow-conn.s.constant}
also works if $s$ is not a constant, as long as $s=o\left( \sqrt{\log(n)} \right)$; it gives a threshold of the form $b_{C_{s,n}} = n^{1-(2-o(1))/s}$ in this case. 

\smallskip

Moving on, let the number of colors be $s = \omega(\log(n))$.
Consider first a system $\mathcal{G}=\{G_1,\ldots,G_s\}$ where each $G_i$ is distributed independently
like $G_{n,p}$. It is proved by Bradshaw and Mohar~\cite{bradshaw2021rainbow}
that the (sharp) threshold for $p$ such that $\mathcal{G}$
is rainbow-connected is $\frac{\log(n)}{sn}$.
That is, for 
$p \leq (1-\varepsilon) \frac{\log(n)}{sn}$ the system $\mathcal{G}$ is not rainbow-connected a.a.s.,
and for 
$p \geq (1+\varepsilon) \frac{\log(n)}{sn}$ the system is rainbow-connected a.a.s.  Now, in contrast to Theorem~\ref{thm:rainbow-conn.s.constant},
it turns out that in this case the game 
$\mathcal{C}_{s,n}$ follows the random graph intuition, 
at least up to a constant factor. 

\begin{theorem}\label{thm:rainbow-conn.s.large}
Let $\log(n) \ll s \leq n$.
The threshold bias for the rainbow-connectivity game 
$\mathcal{C}_{s,n}$ 
on $s$ copies of $K_n$ satisfies
$$
\left(\tfrac{1}{2}-o(1)\right)\frac{sn}{\log(n)} \leq b_{\mathcal{C}_{s,n}} \leq (1+o(1))\frac{sn}{\log(n)} .
$$
\end{theorem}

\bigskip

Another natural and general way to analyse Maker-Breaker games that are played on a colored board is to require Maker to occupy a rainbow winning set of a given positional game. In the case of the connectivity game, Maker's goal is to claim a rainbow spanning tree. 

\begin{definition}[Rainbow-spanning-tree game $\mathcal{RS}_n$] 
The $(1:b)$ game $\mathcal{RS}_{n}$ is played as follows.

Given copies $G_1,\ldots,G_{n-1}$ of the complete graph
$K_n$ on vertex set $V=[n]$, Maker and Breaker alternately claim edges of 
the board $X=\bigcup_{i\in [n-1]}E(G_i)$ according to the given bias. Maker wins iff she manages to claim
a rainbow spanning tree, i.e., a spanning tree on $n$
vertices that contains exactly one edge of 
each of the graphs $G_i$.
\end{definition}

Again, we are able to prove that the game behaves like suggested by the random graph intuition, at least up to a constant factor.

\begin{theorem}\label{thm:rainbow-spanning}
The threshold bias for winning the 
rainbow-spanning tree game $\mathcal{RS}_n$ 
satisfies
$$
\left( \tfrac{\log(2)}{8} - o(1) \right)\frac{n^2}{\log(n)} \leq b_{\mathcal{RS}_{n}} \leq (1+o(1))\frac{n^2}{\log(n)} .
$$
\end{theorem}

\subsection{Diameter game}

Finally, note that the rainbow-connectivity game from Definition~\ref{def:rainbow.conn.game} is closely related to the 
\textit{diameter game} $\mathcal{D}_{s,n}$, 
which was introduced by Balogh et al.~\cite{balogh2009diameter}.
In this $(1:b)$ Maker-Breaker game played on the edges of $K_n$, Maker wins if and only if 
she occupies a spanning subgraph of $K_n$
with diameter at most $s$. That is, she wants to have a path of length at most $s$ between every pair of vertices. Balogh et al.~prove and conjecture the following.

\begin{theorem}[{\cite[Theorem~4]{balogh2009diameter}}]\label{thm:diameter.balogh}
    For every $s\geq 3$ there exist constants
    $c_1,c_2>0$ such that the following holds
    for large enough $n$.
    \begin{enumerate}
        \item[(a)] Maker wins the $(1:b)$ game $\mathcal{D}_{s,n}$ if $b\leq c_1 \left( \frac{n}{\log(n)}\right)^{1-1/\lceil s/2\rceil}$,
        \item[(b)] Breaker wins the $(1:b)$ game $\mathcal{D}_{s,n}$ on $K_n$ if $b\geq c_2 n^{1-1/(s-1)}$.
    \end{enumerate}
\end{theorem}

\begin{conjecture}[Comments after Theorem~4 in~\cite{balogh2009diameter}]\label{conj:balogh.diameter}
The threshold bias $b_{\mathcal{D}_{s,n}}$ is close to the bound for Breaker's side, i.e., the bound in (b).
\end{conjecture}

Adapting parts of our proof for Theorem~\ref{thm:rainbow-conn.s.constant}
we can close the above gap for the diameter game, and disprove 
Conjecture~\ref{conj:balogh.diameter}  for $s\geq 4$.

\begin{theorem}\label{thm:diameter.new}
    Let $s\geq 3$. Then there exist
    constants $c_1,c_2>0$, depending only on $s$,
    such that the following holds
    for every large enough $n$:
    The threshold bias for the diameter game satisfies
    $$
    c_1 n^{1-1/\lceil s/2\rceil}
	\leq 
	b_{\mathcal{D}_{s,n}}    
	\leq
    c_2 n^{1-1/\lceil s/2\rceil}
    .$$
\end{theorem}

\bigskip

\textbf{Organisation of the paper.} 
In Section~\ref{sec:prelim} we collect several useful tools for positional games, probability and rainbow structures. In Section~\ref{sec:rainbow-conn} we
take a closer look at the rainbow connectivity game. Specifically, we first present the proofs of Proposition~\ref{prop:rainbow-2-conn}, Theorem~\ref{thm:rainbow-conn.s.constant}(a)
and Theorem~\ref{thm:rainbow-conn.s.constant}(b). We then complete the proof of  
Theorem~\ref{thm:diameter.new},
and finally also prove Theorem~\ref{thm:rainbow-conn.s.large}. In Section~\ref{sec:rainbow-tree} we consider the rainbow spanning tree game and prove Theorem~\ref{thm:rainbow-spanning}. Finally, we give some concluding remarks and open problems in Section~\ref{sec:concluding}.

\bigskip

\textbf{Notation.}
Some of the more specific notation used in our proofs is defined later, where we need it. Next to this, we use the following standard notation. We write $[n]:=\{k\in \mathbb{N}:~ 1\leq k\leq n\}$.

If a Maker-Breaker game on the edge set of a graph $G$ is in progress, we let $M$ and $B$ denote the subgraphs induced by edges claimed, respectively, by Maker and Breaker. All edges in $M\cup B$ are called \textit{claimed}, while all the other edges of $G$ are called \textit{free}. 

Given a family of events $\mathcal{E}_n$, $n\in \mathbb{N}$,
we say that $\mathcal{E}_n$ holds \textit{asymptotically almost surely} (a.a.s.), if $\Prob(\mathcal{E}_n)$ tends to $1$ when $n$ goes to infinity. We abbreviate \textit{uniformly at random} by u.a.r. Given functions $f,g:\mathbb{N}\rightarrow \mathbb{R}$, we write
$f=O(g)$ if there is a constant $C>0$ such that $|f(n)|\leq C |g(n)|$, for all $n\in \mathbb{N}$. We write $f=\Theta(g)$ if $f=O(g)$ and $g=O(f)$. Moreover, in the case when $\lim_{n\rightarrow \infty} \frac{f(n)}{g(n)}=0$ holds, we use the following notation synonymously:
$f=o(g)$, or $g=\omega(f)$, or $f\ll g$, or $g\gg f$.
Throughout the paper, $\log$ is the natural logarithm.
Moreover, we let $\log^{(1)}(n):=\log(n)$ and $\log^{(k)}(n) :=\log(\log^{(k-1)}(n))$, for all $k\in\mathbb{N}\setminus \{1\}$.

\section{Preliminaries}\label{sec:prelim}

\subsection{Positional Game Tools}

Here we present several standard tools and concepts from positional game theory that will be used in the subsequent analysis.

\subsubsection{Box Games}

For our proofs we make use of the Box Game from~\cite{chvatal1978biased}, and its variant MinBox from~\cite{ferber2015generating}.
The game $\text{Box}(p,1;a_1,\ldots,a_n)$
is a $(p:1)$ Maker-Breaker game played on a hypergraph $(X,\mathcal{H})$,
where the $n$ winning sets in $\mathcal{H}=\left\{F_1,\ldots,F_n\right\}$, called {\em boxes}, are pairwise disjoint
and satisfy $|F_i|=a_i$ for every $i\in [n]$.
The following lemma 
gives a Maker's winning criterion in the case when all the boxes have the same size.

\begin{theorem}[{\cite[Lemma~3.3.2]{hefetz2014positional}}]\label{thm:boxgame}
Let $a_i=m$ for every $i\in [n]$ and assume that $m \leq (p-1) \log(n)$, then Maker wins the game $\mathrm{Box}(p,1;a_1,\ldots,a_n)$, even if Maker is the second player in the game.
\end{theorem}

Similarly, the game $\text{MinBox}(n,D,\gamma,b)$
is a $(1:b)$ Maker-Breaker game 
with $n$ disjoint boxes $F_1,\ldots,F_n$
of size $|F_i|\geq D$, for every $i\in [n]$. But this time, Maker wins if and only if she manages
to claim at least $\gamma|F|$ elements in each of the boxes $F$. 
Ferber et al.~\cite{ferber2015generating}
provide a criterion for Maker to win this game. At any moment of the game, let us call a box $F$
\emph{active}, if Maker still has
less than $\gamma|F|$ of its elements claimed. By applying the methods from~\cite{gebauer2009asymptotic},
Ferber et al.~give a strategy that additionally 
guarantees suitable  bounds on the number of Breaker's elements in all active boxes throughout the game.

\begin{theorem}[{\cite[Theorem~2.5]{ferber2015generating}}]\label{thm:MinBox}
Let $n,D,b$ be positive integers and let $0<\gamma<1$.
Then there is a strategy $\mathcal{S}$ such that Maker can ensure the following to hold throughout the game
$\mathrm{MinBox}(n,D,\gamma,b)$: for every active box $F$ it holds that
$$
w_B(F) \leq b\cdot (w_M(F) + \log(n) + 1), 
$$
where $w_M(F)$ and $w_B(F)$ denote the number of Maker's and Breaker's element in $F$, respectively.
In particular, if $\gamma<\frac{1}{b+1}$ and $D>\frac{b(\ln n +1)}{1-\gamma(b+1)}$,
then Maker wins the game $\mathrm{MinBox}(n,D,\gamma,b)$.
\end{theorem}

\begin{remark}\label{rem:MinBox}
In the strategy $\mathcal{S}$~\cite{ferber2015generating},
in each of her moves, Maker claims an arbitrary element from an active box $F$
for which the \emph{danger}, defined as $\text{dang}(F):=w_B(F)-b\cdot w_M(F)$,
is maximal. 
\end{remark}

\subsubsection{Spooky Balancing Games}

Allen et al.~\cite{allen2017making} introduced a game of similar flavor to the MinBox game,
but with two important differences:
firstly, the winning sets do not need to be disjoint, and secondly, additional winning sets may appear during the game. The latter property can become extremely useful when, for instance, Maker would like to have edges between certain neighborhoods of vertices, but these neighborhoods are created later in the game.
While Allen et al.~call such a game a
Spooky Box Game, we prefer renaming it to
\emph{Spooky Balancing Game} (or \emph{SBG} for short), as the winning sets are allowed to intersect and Maker's goal is not to occupy a winning set completely, but instead claim a reasonable fraction of elements in each of the winning sets.
The prefix ``spooky'' announces an additional (third) player, \emph{Ghost}, utilized to model the possible changes of the winning sets and the board.

\smallskip

In full generality, the 
$(m, b, V, h, \ell, M)$-SBG is played as follows.
In the beginning of the game, we are given hypergraph $\mathcal{H}_0$ with a vertex set $V$  and with exactly $h$ hyperedges, possibly empty at the beginning, which are allowed to grow up to size $M$ during the game.
Initially no vertices are taken by Maker, Breaker or the Ghost. If during the game vertices are taken by Maker or Breaker, we call them \textit{claimed},
and if they are taken by Ghost, we call them \textit{haunted}.
Each round starts with Ghost modifying the current hypergraph $\mathcal{H}_{r-1}$ by adding free vertices from $V$ to the hyperedges of $\mathcal{H}_{r-1}$ in an arbitrary way, but so that their sizes do not exceed $M$. Note that here we call a vertex free if
it is neither claimed by any of the players nor haunted  by Ghost yet.
We call the resulting hypergraph $\mathcal{H}_r$. Afterwards, Maker can repeatedly choose
a free element from $V$ and then
the Ghost may haunt it or let Maker claim this element; additionally Ghost may haunt any further free elements. Maker continues doing so until she claimed $m$ elements in the current round or no further free elements are left.
Afterwards, Breaker claims any $b$ free vertices from $V$.

\smallskip

At any point in the game, let $c_F$ be the number of elements claimed from the winning set $F$ and $c_{F_M}$ be the number of elements of $F$ claimed by Maker. Maker is called the winner of the $(m, b, V, h, \ell, M)$-SBG if at every point in the game and for every winning set $F$, the inequality 
\begin{equation}\label{SBG:Goal}
    c_{F_M} \geq \frac{m}{b+m} \cdot c_F - \ell
\end{equation}
holds. The following theorem provides a criterion for Maker to win the game we just described. 

\begin{theorem}[{\cite[Lemma~4.1]{allen2017making}}]\label{thm:SBG}
Let $m,b,h,M$ be integers and $\ell\in \mathbb{R}$ 
such that
\begin{equation*}
    \text{(1)} ~~ M \geq 9 (m+b)\log(h) ~~~~ \text{  and  } ~~~~
    \text{(2)} ~~ \ell \geq \frac{5mb}{m+b}\sqrt{\frac{M\log(h)}{m + b}}
\end{equation*}
hold. Then Maker has a winning strategy for
the $(m,b,V,h,\ell,M)$-SBG on any starting hypergraph $\mathcal{H}_0$ with vertex set $V$ and $h$ hyperedges.
\end{theorem}

\subsubsection{Beck's Criterion}

We will make use of the following generalization of the Erd\H{o}s-Selfridge Criterion~\cite{erdos1973combinatorial}.

\begin{theorem}[Beck's Criterion, {\cite[Theorem~1]{beck1982remarks}}]\label{thm:beck_criterion}
Let $(X,\mathcal{F})$ be a hypergraph satisfying
$$
\sum_{F\in \mathcal{F}} (1+q)^{-|F|/p} < 1
$$
then Breaker has a strategy to win the $(p:q)$ Maker-Breaker game on $(X,\mathcal{F})$, if he starts the game.
\end{theorem}

\subsubsection{A minimum degree game on multigraphs}

We define the following game, denoted by
$\text{MinDeg}^+(G,d,b,\alpha)$.
The game is a $(1:b)$ Maker-Breaker game on
the edges of a multigraph $G$,
with Breaker being the first player.
Whenever Maker claims an edge, she immediately 
gives it a direction.
Let $d_M^+(v)$ denote Maker's outdegree at a vertex $v$.
Then Maker is said to be the winner of the game if
for every vertex $v\in V(G)$ she can reach
$d_M^+(v)=d$ before $d_B(v)\geq (1-\alpha) d_G(v)$ first happens.

\smallskip

Note that we allow Maker to claim the same edge twice in two different rounds, so that she can give both possible directions to that edge. 
Throughout the game, for every vertex $v$ let
$\text{dang}(v)=d_B(v)-2b\cdot d_M^+(v)$. Moreover, 
we say
that $v$ is active as long as $d_M^+(v)<d$.

Motivated by~\cite{gebauer2009asymptotic},
we consider the following 
strategy $\mathcal{S}^+$:
In every round Maker picks an active vertex of largest danger value, then she claims an arbitrary non-Breaker edge at $v$, which is not already an outgoing Maker-edge at $v$, and makes it an outgoing edge at $v$.

\smallskip

The following lemma can be proven analogously
to~\cite[Lemma~3]{krivelevich2011critical},
or~\cite[Theorem~2.3]{ferber2015generating}.
We postpone its proof to the appendix.

\begin{lemma}\label{lem:degree.game.variant}
For every $\delta>0$ there exists $\eps'>0$ such that
the following holds for every $\eps \in (0,\eps')$ and every large enough $n$.
Let $H_1,H_2$ be vertex-disjoint multigraphs 
on $n$ vertices in which between every pair of vertices there are exactly $s'$ edges. Let
$H=(V(H_1)\cup V(H_2),E(H_1)\cup E(H_2))$
and let $b \leq (1-\delta)\frac{s'n}{\log(n)}$.
Then $\mathcal{S}^+$ is a winning strategy for
$\mathrm{MinDeg}^+(H,\frac{\eps s'n}{b},b,\eps)$.
\end{lemma}

\subsection{Probabilistic Tools}

In our proofs we apply several randomized strategies. In order to analyse these, we make use of the following inequalities that bound the probability that certain random variables deviate too much from their expectation. 
The lemma follows from the well-known Chernoff inequalities (see, e.g.,~\cite{janson2011random}) by a simple coupling argument, see, e.g.,~\cite[Lemma~9]{clemens2025maker}.

\begin{lemma}\label{lemma:Chernoff_modified}
Let $X_1,X_2,\ldots,X_n$ be a sequence of
$n$ Bernoulli random variables 
$X_i\sim Ber(p_i)$ where the $i$-th 
Bernoulli experiment, for every $i\in [n]$, is performed
after and hence independently of
the previous experiments, but where
$p_i$ may depend on the outcome
of these previously performed Bernoulli experiments.
Let $X = X_1+X_2+\ldots+X_n$.
\begin{itemize}
    \item[(a)] If $p_i\geq p$ for every $i\in [n]$, then
    $\Prob[X<(1-\delta)np]<\exp\left(-\frac{\delta^2np}{2}\right)$ for every $\delta>0.$
    \item[(b)] If $p_i\leq p$ for every $i\in [n]$, then 
    $\Prob[X>(1+\delta)np]<\exp\left(-\frac{np}{3}\right)$ for every $\delta\geq 1.$
\end{itemize}
\end{lemma}

\subsection{Rainbow Tools}

For the study of the rainbow spanning tree game, we make use of the following result.

\begin{theorem}\label{thm:rainbow.tree.equiv}
Let $G = (V, E)$ be a multigraph on $n$ vertices and let $c \colon E \rightarrow [n - 1]$ map each edge of $G$ to a color. The following two statements are equivalent.
\begin{enumerate}
    \item[(1)] $G$ has a rainbow spanning tree.
    \item[(2)] Every $k$-partition of $V$ must have edges of at least $k - 1$ distinct colors between the partition classes, for every $k \in [n]$.
\end{enumerate}
\end{theorem}

We note that the above result can be deduced directly from the matroid intersection theorem of Edmonds~\cite[Statement~69]{edmonds2003submodular}, see, e.g., \cite[page~703]{schrijver2003combinatorial}.

\bigskip

\section{Rainbow connectivity game and diameter game} \label{sec:rainbow-conn}

\subsection{Proof of Proposition~\ref{prop:rainbow-2-conn}}

\begin{proof}[Proof of Proposition~\ref{prop:rainbow-2-conn}]
When $b=1$, for every two vertices $v$ and $w$, Maker can pair the two edges between $v$ and $w$.
Whenever Breaker claims one edge, Maker takes simply the other edge of the pair. This way she creates a rainbow path of length 1 between any pair of vertices, and wins.

When $b=2$, let Breaker in the first round claim both edges between a pair of vertices $v,w$ that Maker did not touch in her first move. Afterwards, whenever Maker claims an edge $xv$ (or $xw$) in one of the two colors,
Breaker simply claims $xw$ (or $xv$) in the opposite color, thus preventing Maker from
finishing a rainbow path between $v$ and $w$. That is, Breaker wins.
\end{proof}

\subsection{Proof of Theorem~\ref{thm:rainbow-conn.s.constant}(a)}

\begin{proof}[Proof of Theorem~\ref{thm:rainbow-conn.s.constant}(a)]
Let $C=24s$. By the bias monotonicity
we may assume
that $b = Cn^{\frac{k}{k+1}}$ with $k:=\lceil \frac{s}{2} \rceil-1$. Moreover, whenever needed,
we may assume that $n$ is large enough.

In the following we show that Breaker can play in such a way that at the end of the game there exist two vertices $v,w$ between which Maker has no path of length at most $s$.
That is, we completely ignore colors, and can thus assume to play on a complete multigraph on $n$ vertices where every two vertices are connected by exactly $s$ edges. In order to describe Breaker's strategy, we first introduce some useful notation: at any moment in the game 
we let $V_M$ denote the set of all vertices
that have degree at least 1 in Maker's graph.
Moreover, for every vertex $x$, we let
$C_x$ denote Maker's component containing $x$,
and for every $i\in \mathbb{N}_0$, we let
$$
C^i_x = \{y\in V:~ \text{dist}_M(x,y)=i\}
~~
\text{and}
~~
C^{\leq i}_x = \{y\in V:~ \text{dist}_M(x,y) \leq i\}
$$
denote the set of vertices that
have distance exactly, or respectively, at most $i$ to the vertex $x$ in Maker's graph. Finally we set
$C^{\leq i}:= C_v^{\leq i} \cup C_w^{\leq i}$.

\bigskip

\textbf{Strategy:} After Maker's first move,
Breaker fixes two vertices $v,w$ that Maker did not touch. In each round Breaker then plays as follows:
Assume Maker claimed an edge $xy$ in her last move, then Breaker answers by claiming edges according to the following four points:
\begin{itemize}
\item[(i)] Breaker claims $\frac{b}{4}$ edges incident with $x$ (or all free edges at $x$ if there are fewer),
\item[(ii)] Breaker claims $\frac{b}{4}$ edges incident with $y$ (or all free edges at $y$ if there are fewer),
\item[(iii)] Breaker claims $\frac{b}{4}$ edges iteratively as follows (or fewer if the iteration stops):
choose the smallest $i$ such that there is a free edge intersecting $C^i_v\cup C^i_w$, then claim such a free edge,
\item[(iv)] Breaker claims all free edges 
between $C^{\leq k}$ and $V_M$.
\end{itemize}
If point (iv) requires Breaker to claim more than $\frac{b}{4}$ edges, he forfeits the game.

\bigskip

\textbf{Strategy discussion:} 
Our goal is to show that Breaker never forfeits the game, and that this way Maker is never able to create a path of length at most $s$ between $v$ and $w$. We start with a few simple claims.

\begin{claim}\label{clm:Breaker:conn.bound.close.vertices}
Throughout the game, as long as Breaker does not forfeit, we have
$$|C^{\leq i}_x| \leq 2 \left(\frac{4sn}{b}\right)^i$$
for every $x\in V$ and every $i\in \mathbb{N}_0$.
\end{claim}

\begin{proof}
By (i) and (ii) we get that
$
d_M(z)\leq 
1 + \frac{4sn}{b} 
$
for every vertex $z$ throughout the game.
Therefore, inductively we obtain 
$
|C_x^{j}|\leq \left( 1 + \frac{4sn}{b} \right)^j
$
for every $j\geq 1$, which leads to
$$
|C_x^{\leq i}|
= 
\sum_{j=0}^i |C_x^{j}|
\leq 
\sum_{j=0}^i \left( 1 + \frac{4sn}{b} \right)^j
\leq
2 \left(\frac{4sn}{b}\right)^i 
$$
as claimed.
\end{proof}

\begin{claim}
\label{clm:Breaker:conn.bound.number.rounds}
Assume that until a round $r$
Breaker does not forfeit the game, and that
there is a free edge incident
with $C^{\leq i}$ after round $r$, then
$r < 4 \left(\frac{4sn}{b}\right)^{i+1}$.
\end{claim}

\begin{proof}
We first make the following observation:
If after some round $r'$ there is no
free edge $e$ that is incident with
$C^{\leq i}$,
then Maker cannot further add vertices to the set
$C^{\leq i}$, and hence no such free edge $e$ 
can appear in the following rounds.

Assume now that the claim is wrong,
i.e., we have a round $r \geq 4 \left(\frac{4sn}{b}\right)^{i+1}$ in which there still exists
a free edge incident with 
$C^{\leq i}$.
Then in all previous rounds there 
must have been free edges incident with
$C^{\leq i}$, and hence, by following (iii),
Breaker always claimed $\frac{b}{4}$
edges touching $C^{\leq i}$. Thus,
in total Breaker must have claimed at least
$r \cdot \frac{b}{4} \geq
\left(\frac{4sn}{b}\right)^{i} \cdot 4sn$
edges touching $C^{\leq i}$,
in contradiction to
$|C^{\leq i}|=|C^{\leq i}_v\cup C^{\leq i}_w|\leq 4\left(\frac{4sn}{b}\right)^{i}$.
\end{proof}

Having the above claims in our hand,
we can now show that Breaker never needs more than $\frac{b}{4}$ edges to follow part 
(iv) of the strategy. 

\begin{claim}
\label{clm:Breaker:no.forfeit}
Breaker never forfeits the game.
\end{claim}

\begin{proof}
Assume that until including round $r-1$,
Breaker could always follow the strategy without forfeiting. Assume that in round $r$
Maker claims an edge $xy$.
We want to show that in round $r$, Breaker needs to claim at most $\frac{b}{4}$ free edges when following (iv). 
We consider two cases:
\begin{itemize}
\item Case 1: $x,y\notin C^{\leq k}$ before Maker's move. Then the set $C^{\leq k}$ does not change when Maker claims $xy$. Therefore, to follow part (iv) of the strategy, Breaker only needs to claim free edges between $\{x,y\}$ and $C^{\leq k}$. Due to Claim~\ref{clm:Breaker:conn.bound.close.vertices}, the number of such edges can be upper bounded by
$$
2|C^{\leq k}_v\cup C^{\leq k}_w|
\leq 
8 \left(\frac{4sn}{b}\right)^{k}
<
\frac{b}{4}
$$
where the last inequality holds by the choice of $b$.
\item Case 2: at least one of $\{x,y\}$ is in $C^{\leq k}$ before Maker's move. By symmetry, we may assume that
$x\in C_v^{\leq k}$. Then, since Breaker 
followed part (iv) of the strategy so far,
we have $y\notin V_M$ before Maker's move.
If $y\notin C^{\leq k}_v$ after Maker' move, in order to follow (iv) Breaker only needs to claim all free edges between
$y$ and $C^{\leq k}$. Their number is at most
$$
|C^{\leq k}_v\cup C^{\leq k}_w|
\leq 
4 \left(\frac{4sn}{b}\right)^{k}
<
\frac{b}{4} .
$$
If instead $y\in C^{\leq k}_v$ holds after Maker claims $xy$, then Breaker only needs to claim all free edges between $y$ and the vertices of $V_M$. In this case $x\in C^{\leq k-1}_v$ must hold before Maker's move, and hence $xy$ was a free edge at a vertex in $C^{\leq k-1}_v$. 
By Claim~\ref{clm:Breaker:conn.bound.number.rounds}
this can only happen in the first 
$4 \left(\frac{4sn}{b}\right)^{k}$ rounds, giving that
$$
|V_M|\leq 8 \left(\frac{4sn}{b}\right)^{k} < \frac {b}{4} ,
$$
i.e., Breaker needs to claim at most $\frac{b}{4}$ edges for (iv). \hfill $\qedhere$
\end{itemize}
\end{proof}

It remains to show that whenever Maker 
creates a path between $v$ and $w$,
its length is more than $s$.
So, consider any round $r$ and let $xy$ be the edge that Maker claims in that round.
Assume that this edge closes a path $P$ between $v$ and $w$, and w.l.o.g.~say that $x$ comes before $y$ when going from $v$ to $w$ on that path. Then $x,y\in V_M$ was true before round $r$. Hence, since $xy$ is free at the beginning of round $r$ and Breaker followed part (iv) of the strategy in the first $r-1$ rounds,
we must have $x,y\notin C^{\leq k}$
before Maker claims $xy$.
But this gives $\text{dist}_M(v,x)\geq k+1$
and $\text{dist}_M(y,w)\geq k+1$;
hence the path $P$ has length at least
$2k+3>s$. This completes the proof of the theorem.
\end{proof}

\bigskip

\subsection{Proof of Theorem~\ref{thm:rainbow-conn.s.constant}(b)}

\begin{proof}[Proof of Theorem~\ref{thm:rainbow-conn.s.constant}(b)]
Let $S$ be a set of $s$ colors,
let $k= \left\lfloor \frac{s-1}{2} \right\rfloor$, 
and let $t(s) = 1-1/\lceil \frac{s}{2} \rceil = \frac{k}{k+1}$. 
Further, choose $\delta = 10^{-80s}s^{-20}$
and $\gamma = 10^{-4s}\delta$.
Whenever necessary, we assume that $n$ is large enough.
Let $G =(V,E)$ be the multigraph on which the game is played, consisting of $s$ edge-disjoint copies $G_c$ of $K_n$, $c\in S$, on the same vertex set $V$. 
Hence, each graph $G_c$ is identified with a distinct color $c\in S$, and for short, we sometimes
call it the layer $c$.
By bias monotonicity assume that Maker plays against a bias $b = \gamma n^{t(s)}$.  
We set $p = n^{-t(s)}$ and $\alpha_1 = \frac{1}{50 b}$ and $\alpha_2 = 2\alpha_1$.

\medskip

To introduce the idea behind Maker's strategy, 
we fix a pair of vertices $v, w \in V$ and a tuple $C = (c_1, \dots, c_{2k+1}) \in S^{2k+1}$. 
The goal is to connect $v$ and $w$ via a path of length $2k+1$ 
that uses edges of the colors $C$ in the given order.
We will ensure that Maker's graph expands $k$ steps from $v$ 
using edges of colors $c_1, \dots, c_k$ in that order
and expands $k$ steps from $w$ 
using edges of colors $c_{2k+1}, c_{2k}, \dots, c_{k + 2}$, again in that order.
The two sets reached by this expansion are then connected via final edges in color $c_{k + 1}$, finishing the path. 

In order to ensure that the two expanding sides never overlap, 
we partition the vertex set into $V = V_L \cup V_R$ 
with $\lvert V_L \rvert = \lvert V_R \rvert = \frac n 2$.
Regardless to which halves choices of $v$ and $w$ belong themselves, we will use expansion of $v$ into $V_L$ and  expansion of $w$ into $V_R$.
This expansion property of every vertex into both halves is guaranteed by the edges that Maker claims within a random multigraph $\Gamma_1$ generated during the game, 
while the connection between the expanded sets uses Maker's edges of a second random multigraph $\Gamma_2$.
Both random graphs can be split into their individual layers, 
where $\Gamma_{1, c}$ (respectively $\Gamma_{2, c}$) 
denotes the subgraph using exactly the edges of $\Gamma_1$ (respectively $\Gamma_2$) 
in color $c$.
Maker generates these multigraphs during the game and guarantees 
to claim enough of their edges to a.a.s.~guarantee the properties sketched above 
on the Maker submultigraphs $\Gamma_1 \cap M$ and $\Gamma_2 \cap M$.

The random graph $\Gamma_{1,c}$ is generated 
while playing a degree game in the rough style of~\cite{krivelevich2011critical}:
At every vertex $v\in V$ and for every layer $c \in S$, 
two vertex sets are selected, both of size $\alpha_2 \frac n 2$. 
The first, call it $\NeighsOutL{\Gamma_1}{c}{v}$, 
is chosen u.a.r.~as subset of $V_L \setminus \{v\}$. 
The second, call it $\NeighsOutR{\Gamma_1}{c}{v}$, 
is chosen  u.a.r. as subset of $V_R \setminus \{v\}$.
All edges in color $c$ from $v$ to a vertex in $\NeighsOutL{\Gamma_1}{c}{v} \cup \NeighsOutR{\Gamma_1}{c}{v}$ 
are added to $\Gamma_{1,c}$. 
For our later analysis,
we imagine these edges to be directed  
and pointing out from $v$
and we say that they are \emph{exposed} at vertex $v$. 
Note that it may happen that an edge $xy$ in $\Gamma_{1,c}$ gets exposed at both vertices $x$ and $y$; then we still add $xy$ to $\Gamma_{1,c}$ only once, but give it both possible directions.

The second random graph $\Gamma_{2,c}$ is generated
while playing a degree game similarly to~\cite{ferber2015generating} and a balancing game in the style of~\cite{allen2017making}, and it is going to have distribution $G_{n,p}$, i.e., every edge appears independently with probability $p$.
For both these random graphs $\Gamma_{1,c}$ and
$\Gamma_{2,c}$,
Maker exposes the edges while the game proceeds, in an order determined by her strategy, and whenever possible, i.e., when an edge that is revealed to be in $\Gamma_{1,c}$ or $\Gamma_{2,c}$ is still free, she claims that edge. Moreover, if an edge of layer $c$ is revealed not to be in $\Gamma_{2,c}$ we add it to a set $\text{Sp}$ (of spooky edges that the Ghost haunts) which is initially empty.

\medskip

Before describing the strategy in detail, we introduce some more notation: For each vertex $v$ and every layer $c$, we denote the set of neighbors of $v$ in layer $c$ as $\Neighs_c(v)$, 
and for any $H\subseteq G$, we let $\Neighs_{H,c}(v)$ be the neighbors of $v$ in $H$ that belong to layer $c$.
Accordingly, we define
degrees $\Degr_c(v)=\abs{\Neighs_c(v)}$
and
$\Degr_{H,c}(v)=\abs{\Neighs_{H, c}(v)}$.
The following definitions are done once 
for each of the two vertex partition classes $V_L$ and $V_R$. 
Thus, let $D \in \{L, R\}$.
The subset of $\NeighsKOut{D+}{\Gamma_1}{c}{v}$
containing the vertices $w$, such that the edge $vw$ of color $c$ belongs to Maker's subgraph of $\Gamma_{1, c}$, is called $\NeighsKOut{D+}{M}{c}{v}$.
We set
$\NeighsKOut{D-}{\Gamma_1}{c}{v} =\{w\in V: v \in \NeighsKOut{D+}{\Gamma_{1}}{c}{w}\}$.
The corresponding set in Maker's subgraph is defined as 
$\NeighsKOut{D-}{M}{c}{v} = \{w\in V: v \in \NeighsKOut{D+}{M}{c}{w}\}$.
Accordingly,
we define
$\DegrKOut{D+}{M}{c}{v} = \abs{\NeighsKOut{D+}{M}{c}{v}}$,
$\DegrKOut{D-}{M}{c}{v} = \abs{\NeighsKOut{D-}{M}{c}{v}}$,
$\DegrKOut{D+}{\Gamma_1}{c}{v} = \abs{\NeighsKOut{D+}{\Gamma_1}{c}{v}}$ and
$\DegrKOut{D-}{\Gamma_1}{c}{v} = \abs{\NeighsKOut{D-}{\Gamma_1}{c}{v}}$.
Moreover, for a set $A\subset V$, we define 
$\NeighsKOut{D+}{M}{c}{A} = \bigcup_{v\in A} \NeighsKOut{D+}{M}{c}{v}$ and $\DegrKOut{D+}{M}{c}{A} = \abs{\NeighsKOut{D+}{M}{c}{A}}$.

Let $\Phi = S^{2k+1}$. 
For $C \in \Phi$ with $C = (c_1, \dots, c_{2k+1})$, we write $C_L = (c_1, \dots, c_k)$
and $C_R = (c_{2k+1}, c_{2k}, \dots, c_{k + 2})$.
We will often refer to $C$ as $(C_L, c, C_R)$, where $c = c_{k + 1}$.
Note that $\abs{\Phi} = s^{2k+1}\leq s^s$.
For a vertex $v$ and a sequence of colors $C=(c_1,\ldots,c_t) \in S^t$, let 
$\NeighsKOut{D+}{\Gamma_1}{C}{v}$ be the set of all vertices $v_{t}$, that can be reached from $v$ by a path $ve_1v_1e_2v_2...e_{t} v_{t}$ where for each $i\in [t]$, the edge $e_i$ has the color $c_i$ and is oriented towards the vertex $v_i \in V_D$ in the orientation imagined on $\Gamma_{1,c_i}$. Similarly, let
$\NeighsKOut{D+}{M}{C}{v}$ be the subset of all such vertices $v_{t}$ where the described path $ve_1v_1e_2v_2...e_{t} v_{t}$ is such that $e_i$ is claimed by Maker for every $i\in [t]$.
Again, let $\DegrKOut{D+}{M}{C}{v} = |\NeighsKOut{D+}{M}{C}{v}|$.

Finally, whenever we define an operator only considering the edges of Maker, we take the freedom to replace the corresponding $M$ with a $B$ to only consider the edges of Breaker, e.g. $\NeighsKOut{L}{B}{c}{v}$ denotes the subset of $V_L$ that is adjacent to $v$ in color $c$ in Breaker's graph.

\bigskip

\textbf{Strategy:}
In her strategy, Maker plays three different subgames in parallel, leading to three different types of moves, where moves of type $i\in \{1,2,3\}$ are always performed in rounds 
$r \equiv i$ (mod $3$). Since we do not control Breaker's moves in between two moves of the same type, we just assume that Maker plays these subgames against bias $3b$. We start by introducing these three subgames. In the strategy discussion we then show that Maker can a.a.s.~ensure to reach all her goals from these different subgames, and building on that, we prove that Maker wins a.a.s. The subgames 
are as follows:

\medskip\medskip

\begin{minipage}{30pt}
~~~~ 
\end{minipage}
\begin{minipage}{0.9\textwidth}
\begin{itemize}
    \item[Type 1:] minimum degree game w.r.t.~the edges of $\Gamma_1$
    \item[Type 2:] minimum degree game w.r.t.~the edges of $\Gamma_2$
    \item[Type 3:] balancing game w.r.t.~the edges of $\Gamma_2$
\end{itemize}
\end{minipage}

\medskip\medskip

\noindent
\underline{Type 1:} While the game proceeds,
Maker imagines to play the 
$\left(2sn, \frac n 2, \alpha_1, 6b\right)$-MinBox game
in parallel. Here, for each color $c\in S$ and for each vertex $v\in V$, we create two boxes $B_{v, c}^L$ and $B_{v, c}^R$ of size $\frac n 2$, and Maker wins the game
if she claims at least $\alpha_1 \frac n 2 = \frac{n}{100 b}$ elements in each box.

\smallskip

In the actual game on $G$, a move of type 1 is done as follows. At first Maker updates the simulated $\left(2sn, \frac n 2, \alpha_1, 6b\right)$-MinBox game: for each edge $xy$ that Breaker claimed in some layer $c$ since Maker's last move of type 1, we assume that Breaker claimed an element of the box $B_{x, c}^{D_1}$, 
where $D_1 \in \{L, R\}$ is such that $y \in V_{D_1}$, 
and an element of the box $B_{y, c}^{D_2}$, 
where $D_2 \in \{L, R\}$ is such that $x \in V_{D_2}$. 

Afterwards, in the simulated MinBox game,
Maker follows the strategy $\mathcal{S}_1$ from 
Lemma~\ref{thm:MinBox} (see also Remark~\ref{rem:MinBox}) and accordingly she claims an element of a box $B_{v, c}^{D}$ for some vertex $v$, some color $c$ and some $D \in \{L, R\}$. 
Having $v$, $c$ and $D$ fixed this way, in the actual game on $G$,
Maker chooses u.a.r.~an edge $vw$ to a vertex $w \in V_D$ that is incident to $v$ in layer $c$ and was not exposed at the vertex $v$ yet. 
Maker adds the edge $vw$ to $\Gamma_{1,c}$ and imagines it being directed from $v$ to $w$
from now on. The following cases may happen: 
\begin{itemize}
\item If this random edge leads 
$\Gamma_{1,c}$ to have more than $\alpha_2 \frac n 2$ outgoing edges at $v$ into $V_{D}$, Maker forfeits the game. 

\item Otherwise, if Breaker has not claimed the edge $vw$ in layer $c$ yet, Maker claims it. (It may happen that $vw$ already belongs to Maker's graph due to another move of type 1,~2 or 3. Then Maker simply imagines to claim the edge again.)

\item Otherwise, if Breaker has claimed the edge $vw$ already, Maker repeats the process with another random edge at $v$ in layer $c$ into $V_D$.
\end{itemize}

If immediately after Maker's move we have $\DegrKOut{D+}{M}{c}{v} = \alpha_1 \frac n 2$ (i.e., in the simulated MinBox game, the box $B_{v, c}^{D}$ is not active from now on),
she exposes further edges from $v$ into $V_{D}$ in layer $c$
as follows. Let $r\leq \alpha_2 \frac n 2$ be the number of edges from $v$ into $V_{D}$ which were added to $\Gamma_{1,c}$ so far, then Maker chooses u.a.r.~$\alpha_2 \frac n 2 -r$ edges from $v$ into $V_{D}$ in layer $c$ among those edges which were not exposed at $v$ yet. She adds all these edges to $\Gamma_{1,c}$ (but she does not claim these edges).

\medskip

\underline{Type 2:} 
While the game proceeds,
Maker imagines to play the 
$(1,3b,E(G),sn,\ell_2,M_2)$-SBG
in parallel, where
$$
M_2 = n-1 ~~ \text{and} ~~
\ell_2 = \delta pn .
$$
Here, for each color $c\in S$ and for each vertex $v\in V$, we have a hyperedge $F_v^c$ of size $n-1$ 
that consists of the edges incident
with $v$ in layer $c$.  

\smallskip

In the actual game on $G$, a move of type 2 is done as follows. At first Maker updates the simulated $(1,3b,E(G),sn,\ell_2,M_2)$-SBG: for each edge that Breaker claimed since the last Maker's move of type 2, we assume that Breaker claimed the same edge in the Spooky Balancing Game, except if this edge belongs to $\text{Sp}$ (and hence belongs to Ghost in the SBG).
Afterwards, in this simulated SBG,
Maker follows the strategy $\mathcal{S}_2$ from 
Theorem~\ref{thm:SBG} 
(we show in Claim~\ref{clm:rainbow-conn.SBGcheck} that $\mathcal{S}_2$ is a winning strategy)
and accordingly she chooses an edge $xy$ of $G$; say this edge 
belongs to layer $c$ for some $c\in S$. Having $xy$ 
fixed this way, Maker does the following in the actual game on $G$.
\begin{itemize}
\item If Maker already tossed a coin on the edge $xy$ in layer $c$
in a previous round (due to a move of type 3), then two cases may happen:
\begin{enumerate}
\item[(a)] If the coin toss was a success,
she imagines claiming the edge $xy$ in layer $c$ in the current round. Accordingly, Maker also claims the edge in the simulated SBG.
\item[(b)] If the coin toss was not a success, Maker does not claim the edge $xy$ in layer $c$. Accordingly, in the simulated SBG the edge is haunted by Ghost. Then Maker 
repeats the process with 
choosing the next edge proposed by the strategy $\mathcal{S}_2$, until she finally claims an edge in this round.
\end{enumerate}
\item If Maker did not already toss a coin on the edge $xy$ in layer $c$
in a previous round, she now tosses a coin 
with $p$ being the probability of success. Again two cases may happen:
\begin{enumerate}
\item[(a)] If the coin toss is a success,
Maker claims the edge $xy$ in layer $c$,
and she adds $xy$ to the random graph $\Gamma_{2,c}$. Accordingly, Maker claims the edge in the simulated SBG.
(It may happen that $xy$ already belongs to Maker's graph, e.g., due to a move of type 1. Then Maker simply imagines to claim the edge again.)
\item[(b)] If the coin toss is not a success, Maker does not claim the edge,
and it is revealed not to be part of $\Gamma_{2,c}$. To remember this fact,
we add the edge to the set $\text{Sp}$. Accordingly, in the simulated SBG the edge is haunted by Ghost. Then Maker repeats the whole process with 
choosing the next edge proposed by the strategy $\mathcal{S}_2$, until she finally claims an edge in this round.
\end{enumerate}
\end{itemize}

\medskip

\underline{Type 3:} 
While the game proceeds,
Maker imagines to play the 
$(1,3b,E(G),s^{2k+1} n^2,\ell_3,M_3)$-SBG
in parallel,
where
$$
M_3 = \left( \frac{p n}{100\gamma} \right)^{2 k}
~~ \text{and} ~~
\ell_3 = \delta p\left(p n\right)^{2 k} .
$$
Here, for each choice of $v,w\in V$ 
and $(C_L,c,C_R)\in \Phi$,
we have a winning set $B_c(v,w,C_L, C_R)$,
which is empty at the beginning
and to which Ghost may add
free edges of layer $c$ 
that turn out to lie between
$\NeighsOutL{M}{C_L}{v}$ and
$\NeighsOutR{M}{C_R}{w}$ when the two sets 
increase.

\smallskip

In the actual game on $G$, a move of type 3 is then done similarly to moves of type 2. Maker updates the simulated $(1,3b,E(G),s^{2k+1} n^2,\ell_3,M_3)$-SBG: 
at first, for each edge that Breaker claimed since the last Maker's move of type 3, we assume that Breaker claimed the same edge in the Spooky Balancing Game, except if the edge belongs to $\text{Sp}$.
Secondly, if for some $v,w\in V$ and $(C_L,c,C_R)\in \Phi$, it happens that
$\NeighsOutL{M}{C_1}{v}$ or $\NeighsOutR{M}{C_2}{w}$
increased in the meantime,
Ghost adds all free edges of layer $c$ that suddenly lie between
$\NeighsOutL{M}{C_L}{v}$ and $\NeighsOutR{M}{C_R}{w}$
to the set $B_c(v,w,C_L,C_R)$.
If this causes $\abs{B_c(v,w,C_L,C_R)} > M_3$, Maker forfeits the game.

Afterwards, in this simulated SBG,
Maker follows the strategy $\mathcal{S}_3$ from 
Theorem~\ref{thm:SBG} 
(we show in Claim~\ref{clm:rainbow-conn.SBGcheck} that $\mathcal{S}_3$ is a winning strategy)
and accordingly she chooses an edge $xy$ of $G$; say this edge 
belongs to layer $c$ for some $c\in S$. Then Maker does the same case distinction and coin toss as in moves of type 2,
where this time she uses
the $(1,3b,E(G),s^{2k+1} n^2,\ell_3,M_3)$-SBG
together with the strategy $\mathcal{S}_3$.

\bigskip

\underline{End of game:} Maker stops playing the game, when each edge in $G$
belongs to Maker, or belongs to Breaker,
or was revealed to be neither in $\Gamma_1$ nor in $\Gamma_2$. For every edge $xy$ belonging to Breaker in some layer $c$ which was not exposed by type 2 or 3 so far (and hence does not belong to $\text{Sp}$), Maker finally tosses a coin independently and with success probability $p$, 
and in case of success, she adds
$xy$ to $\Gamma_{2,c}$.

\bigskip

\textbf{Strategy discussion:}
Whenever needed, assume that $n$ is large enough. We start by showing that the parameters for our Spooky Balancing Games are chosen in such a way that Theorem~\ref{thm:SBG} can be applied.

\begin{claim}\label{clm:rainbow-conn.SBGcheck}
Maker wins the Spooky Balancing Games she plays in moves of type 2 and 3.
\end{claim}
\begin{proof}
To prove the claim, we need to check the two inequalities (1) and (2) of Theorem \ref{thm:SBG}. 
We start with the strategy $\mathcal{S}_2$, so with the $(1, 3b, E(G), sn, \ell_2, M_2)$-SBG. Inequality (1) holds in that case since $M_2 = \Theta(n)$ and $9(1+3b)\log(sn) = o(n)$. Inequality (2) holds too, since
$$ \frac{15b}{1 + 3b} \sqrt{\frac{M_2 \log(sn)}{1 + 3b}} = O\!\left(\sqrt\frac{n\log(n)}{n^{t(s)}}\right) = O\!\left(n^\frac{1}{2(k+1)} \sqrt{\log(n)}\right) = o\left(n^\frac{1}{k+1}\right) = o(\ell_2).$$
Now, we look at the strategy $\mathcal{S}_3$, hence at the $(1,3b,E(G),s^{2k+1} n^2,\ell_3,M_3)$-SBG. It holds that $M_3 = \Theta\left(n^{2k/(k+1)}\right) = \Theta\left(b^2\right)$, while $9(1 + 3b)\log(s^{2k+1} n^2) = O(b\log(n))$, which proves inequality~(1). Inequality (2) holds as well, since
$$\frac{15b}{1 + 3b}\sqrt{\frac{M_3\log(s^{2k+1} n^2)}{1 + 3b}} =O\left(\sqrt{p(pn)^{2k}\log(n)}\right) = o\left(p(pn)^{2k}\right) = o\left(\ell_3\right).$$
Thus, the claim holds for large enough $n$.
\end{proof}

Next, we want to show that Maker a.a.s.~does not forfeit the game. For this, we first prove that, when following the moves of type 1, Maker never runs out of free edges.

\begin{claim}\label{clm:rainbow-conn.Breaker-degree}
    As long as Maker does not forfeit: The moves of type 1 ensure that for every vertex $v$, every color $c$ and every $D \in \{L, R\},$ 
    as long as $\DegrKOut{D+}{M}{c}{v} < \frac{n}{100b}$ holds, the Breaker-degree is bounded by $\DegrKOut{D}{B}{c}{v} < 0.14 \frac n 2$. 
\end{claim}

\begin{proof}
If $\DegrKOut{D+}{M}{c}{v} < \frac{n}{100b}$,
then in the simulated MinBox game,
Maker's number of elements in the box $B_{v, c}^D$ is $w_M(B_{v, c}^D) = \DegrKOut{D+}{M}{c}{v} < \alpha_1 \frac n 2 = \frac{n}{100 b}$, and hence the box $B_{v, c}^D$ is active.
By the way the MinBox game is updated in moves of type 1, we have that Breaker's number
of elements in this simulated game satisfies $w_B(B_{v, c}^D) = \DegrKOut{D}{B}{c}{v}$. Therefore, Theorem~\ref{thm:MinBox} implies
$$
\DegrKOut{D}{B}{c}{v}
\leq 6b (\DegrKOut{D+}{M}{c}{v} + \log(2 s n)+1) < 0.14 \frac n 2
$$
for large enough $n$, where the last inequality uses the assumed bound on
$\DegrKOut{D+}{M}{c}{v}$.
\end{proof}

Hence, as long as Maker needs to expose and claim edges of $\Gamma_{1,c}$ at a vertex $v$
into a set $V_D$,
we can assume that at least $0.86 \frac{n}{2}$ edges incident with $v$ in layer $c$ into $V_D$ are not claimed by Breaker.
We will make use of this observation many times when bounding probabilities that failures happen or certain bad situations (to be described later) occur too often. 

\begin{claim}\label{clm:rainbow-conn.no.forfeit}
    A.a.s. Maker does not forfeit the game.
\end{claim}

\begin{proof}
Maker would forfeit the game in two cases. 
The first case occurs, if a move of type 1 requires her to add more than $\alpha_2 \frac n 2$ outgoing edges from some vertex $v$ in some color $c$ into some subset $V_D$ to $\Gamma_1$.
The second case occurs, when the game state for a move of type 3 can no longer be converted into the SBG setting, because a 
hyperedge $B_c(v,w,C_L,C_R)$ becomes too large.

We will show, that the first case a.a.s. never occurs.
If Maker exposes an edge, that is already occupied by Breaker, 
we call it a \emph{failure}. 
For each vertex $v$, every color $c$ and every $D \in \{L, R\}$,
let $F_c^D(v)$ be the number of \emph{failures} Maker makes, while exposing edges of color $c$ from $v$ into $V_D$ before $\DegrKOut{D+}{M}{c}{v} = \alpha_1 \frac n 2$ is reached. We need to show that $F_c^D(v) \leq \alpha_2 \frac n 2  - \alpha_1 \frac n 2 = \alpha_1 \frac n 2$ holds for every $c \in S$ and $v \in V$.
    
From the previous claim, we know that for every vertex $v$ and every color $c$, 
the Breaker-degree into $V_D$, denoted $\DegrKOut{D}{B}{c}{v}$ 
is bounded from above by $0.14 \frac n 2$, 
as long as $\DegrKOut{D+}{M}{c}{v} < \alpha_1 \frac n 2$ holds. 
Hence, we can bound the probability of a specific try being a failure by 
$\frac{0.14 \frac n 2}{(1 - \alpha_2) \frac n 2} < 0.15$ 
as long as Maker has not exposed more than $\alpha_2 \frac n 2$ edges and $\DegrKOut{D+}{M}{c}{v} < \alpha_1 \frac n 2$.
Thus, while doing at most $\alpha_2 \frac n 2$ exposes at $v$ in layer $c$ into $V_D$, 
the expected number of failures is at most $0.15 \alpha_2 \frac n 2 = 0.3 \alpha_1 \frac n 2$. 
Using Chernoff (Lemma \ref{lemma:Chernoff_modified}) and a union bound 
it follows that $F_c^D(v) \leq \alpha_1 \frac n 2$ holds a.a.s.~for every vertex $v$, every color $c$ and every side $D$.

\smallskip

Left to show is that Maker will also not forfeit in moves of type 3.
By construction, for every vertex $v$, every color $c$ and every side $D$, 
we have $\DegrKOut{D+}{M}{c}{v} \leq \alpha_1 \frac n 2$. 
Thus, for all $(C_L, c, C_R) \in \Phi$ and vertices $v$ and $w$,
the bounds
$$
    \DegrOutL{M}{C_L}{v} 
    \leq \left(\alpha_1 \frac n 2\right)^k 
    = \left(\frac{p n}{100\gamma}\right)^k
    \text{~~~and~~~}
    \DegrOutR{M}{C_R}{w} \leq \left(\frac{p n}{100\gamma}\right)^k 
$$
hold throughout the game.
As exactly $\DegrOutL{M}{C_L}{v} \cdot \DegrOutR{M}{C_R}{w}$ edges in color $c$ exist between $\NeighsOutL{M}{C_L}{v}$ and $\NeighsOutR{M}{C_R}{w}$, 
the requirement $\abs{B_c(v, w, C_L, C_R)} \leq M_3$ is never violated.
\end{proof}

Next, we turn our attention to the random graphs that Maker generates, and prove a few useful properties that we need for later claims.

\begin{claim}\label{clm:rainbow-conn.random.graphs}
If Maker does not forfeit the game,
for each $c\in S$, the random graph $\Gamma_{2,c}$ is distributed like $G_{n,p}$.
Moreover, the following properties hold a.a.s.:
\begin{enumerate}
    \item[(i)] $\DegrKOut{D-}{\Gamma_1}{c}{v} \leq \frac{n}{b}$ 
        for every $D \in \{L, R\}$, $v\in V_D$ and $c\in S$,
    \item[(ii)] $(1-\delta)np \leq \Degr_{\Gamma_{2, c}}(v) \leq (1+\delta)np$ for every $v\in V$ and $c\in S$,
    \item[(iii)] For every $v,w\in V$ and 
    $(C_L,c,C_R)\in \Phi$ the following holds:
    for all subsets $A \subseteq \NeighsOutL{\Gamma_1}{C_L}{v}$ 
    and 
    $B \subseteq \NeighsOutR{\Gamma_1}{C_R}{w}$ 
    with
    $|A|\geq \frac{1}{10^3} \left(\frac{n}{10^4b} \right)^{k}$
    and
    $|B|\geq \frac{1}{10^3} \left(\frac{n}{10^4b} \right)^{k}$, we have
    $$
    \frac{1}{2}p|A||B| \leq 
    e_{\Gamma_{2,c}}(A,B) \leq 2p|A||B| .
    $$
\end{enumerate}
\end{claim}

\begin{proof}
The random graph $\Gamma_{2,c}$
is generated in moves of type 2 and 3.
In these moves, until the end of the game, Maker tosses her coin exactly once on each edge that does not belong to $B\setminus \text{Sp}$. Then, when the game is already over, she also tosses her coin on each edge in $B\setminus \text{Sp}$, which ensures that $\Gamma_{2,c}$ is distributed like $G_{n,p}$. It thus remains to verify that (i)--(iii) hold a.a.s.

We start with property (i).
For each vertex $w \in V$ with $w \neq v$, the probability that there is an oriented edge $(w,v)$ of color $c$ is at most
$\frac{\alpha_2 \frac{n}{2}}{\frac{n}{2}-1} = (1+o(1))\alpha_2$,
and hence in expectation $\DegrKOut{D-}{\Gamma_1}{c}{v}$ is of size $(1+o(1))\alpha_2n$. 
Using Chernoff (Lemma~\ref{lemma:Chernoff_modified})
together with a union bound 
over all $D\in\{L,R\}$, $v\in V_D$ and $c\in S$
it follows that (i) holds a.a.s.

\smallskip

Next, consider property (ii). 
Since $\Gamma_{2,c}$ is distributed like
$G_{n,p}$, we have that
$\Degr_{\Gamma_{2,c}}(v)$ is distributed like
$Bin({n - 1},p)$ with expected value
$(n-1)p$. Again, using a Chernoff inequality (Lemma~\ref{lemma:Chernoff_modified})
together with a union bound 
over all $v\in V$ and $c\in S$
it follows that (ii) holds a.a.s.

\smallskip

It remains to check (iii).
It is enough to prove that
for a fixed choice of $v,w\in V$ and 
$(C_L, c, C_R)\in \Phi$,
the desired property in (iii) fails  with
probability at most $\exp(-\sqrt{n})$.
Then by a union bound over all $s^{2k+1} n^2$
options to choose $v,w\in V$ and 
$(C_L, c, C_R)\in \Phi$ we can conclude that (iii) holds a.a.s.

Hence, let $v,w\in V$ and 
$(C_L, c, C_R)\in \Phi$ be fixed from now on.
We can look at the relevant random graph
without knowing what Maker and Breaker do.
So, let us first reveal the random graphs $\Gamma_{1,c'}$ with $c'\in S$. 
Then 
$\DegrKOut{D+}{\Gamma_1}{c'}{x} = \alpha_1 \frac{n}{2} = \frac{n}{100b}$
for every $x\in V$, $c'\in S$ and $D \in \{L, R\}$,
by definition of $\Gamma_{1,c'}$,
which implies
\begin{equation}\label{eq:random.graphs.bounds1}
    \abs{\NeighsOutL{\Gamma_1}{C_L}{v}} \leq \left( \frac{n}{100b} \right)^{k}
~~
\text{and}
~~
\abs{\NeighsOutR{\Gamma_1}{C_R}{w}} \leq \left( \frac{n}{100b} \right)^{k}.
\end{equation}

Only afterwards, reveal the random graph $\Gamma_{2,c}$.
Then for every 
$A \subseteq \NeighsOutL{\Gamma_1}{C_L}{v}$ and 
$B \subseteq \NeighsOutR{\Gamma_1}{C_R}{v}$
of size $|A|,|B| \geq \frac{1}{10^3} \left(\frac{n}{10^4b}\right)^{k}$, we have
$\Exp[ e_{\Gamma_2, c}(A, B) ] = p|A||B|$,
as $\Gamma_{2,c}$ is distributed like $G_{n,p}$ and since
$A\subseteq V_L$
and $B \subseteq V_R$ implies $A \cap B = \varnothing$.
Applying Chernoff (Lemma~\ref{lemma:Chernoff_modified}), 
we get 
$$ \mathbb{P}\left( \frac{1}{2}p|A||B| \leq 
    e_{\Gamma_{2,c}}(A,B) \leq 2p|A||B| ~~ \text{fails}\right) 
    \leq \exp\left(-\frac{1}{100}p|A||B|\right).
$$
Thus, by a union bound over all 
$A \subseteq \NeighsOutL{\Gamma_1}{C_L}{v}$ 
and 
$B \subseteq \NeighsOutR{\Gamma_1}{C_R}{w}$ 
with $|A|,|B| \geq \tilde{a} := \frac{1}{10^3} \left(\frac{n}{10^4b}\right)^{k}$, we see that (iii) fails for our fixed choice of $v,w\in V$ and $(C_L, c, C_R)\in \Phi$ with probability at most 
    \begin{align*}
         & \sum_{\substack{A \subseteq \NeighsOutL{\Gamma_1}{C_L}{v}\\ |A| \geq \tilde{a}}}
            ~~\sum_{\substack{B \subseteq \NeighsOutR{\Gamma_1}{C_R}{w} \\ |B| \geq \tilde{a}}} 
            \exp\left(-\frac{1}{100}p\abs{A}\abs{B}\right) 
        \leq ~ 2^{|\NeighsOutL{\Gamma_1}{C_L}{v}|} \cdot 
            2^{|\NeighsOutR{\Gamma_1}{C_R}{w}|} \cdot 
            \exp\left(-\frac{1}{100}p\tilde{a}^2\right) \\
        \leq ~& \exp\left( 2\left(\frac{n}{100b}\right)^{k}- 
            \frac{1}{10^8} n^{-\frac{k}{k + 1}} \left(\frac{n}{10^4b}\right)^{2 k}\right) 
        =  \exp \Bigg(\left(\frac{2}{100^k \gamma^k} 
            - \frac{1}{10^{8 (k+1)} \gamma^{2 k}}\right)
            n^{\frac{k}{k + 1}}\Bigg) \\
        \leq ~& \exp \left( - n^{\frac{k}{k + 1}} \right)
        \leq \exp(-\sqrt{n}) 
    \end{align*}
where the second inequality uses 
\eqref{eq:random.graphs.bounds1}, the only equation uses the definition
of $b$, and the last line uses the choice of $\gamma$. 
This finishes the proof of the claim.    
\end{proof}

Next, we prove degree and expansion properties
for Maker's subgraphs of $\Gamma_{1,c}$.
The proof for part (iii) follows an argument of Krivelevich~\cite{krivelevich2011critical}.

\begin{claim}\label{clm:rainbow-conn.expansion}
    A.a.s.~the moves of type 1 ensure that at the end of the game the following properties are true for all vertices $v$, all colors $c$, all $D \in \{L, R\}$ and all sets of vertices $A$ with $|A| \leq b$:
    $$\text{(i)}~ 
    \DegrKOut{D+}{M}{c}{v} = \frac{n}{100b},
    \quad\quad\quad
    \text{(ii)}~\DegrKOut{D-}{M}{c}{v} \leq \frac{n}{b},
    \quad\quad\quad
    \text{(iii)}~ \DegrKOut{D+}{M}{c}{A} \geq \frac{|A|n}{10^4b}.$$
\end{claim}

\begin{proof}
Condition on the good events of Claim~\ref{clm:rainbow-conn.no.forfeit}
and Claim~\ref{clm:rainbow-conn.random.graphs}.
We start with (i). The equation $\DegrKOut{D+}{M}{c}{v} = \alpha_1 \frac n 2 = \frac{n}{100b}$ holds for every vertex $v$, every color $c$ and every side $D$, since Maker wins the MinBox-game, she plays during moves of type 1. 
By Claim~\ref{clm:rainbow-conn.Breaker-degree},
we know that, as long as $\DegrKOut{D+}{M}{c}{v} < \alpha_1 \frac n 2$ holds, $\DegrKOut{D}{B}{c}{v} \leq 0.14 \frac n 2$ holds as well. Hence, the box $B_{v, c}^D$ is not full and Maker can continue claiming elements of that box and hence outgoing edges at $v$
until $\DegrKOut{D+}{M}{c}{v}$ is equal to $\alpha_1 \frac n 2$. Afterwards Maker stops claiming edges in layer $c$ from $v$ into $V_D$, due to Remark~\ref{rem:MinBox}.
For (ii), note that Maker's directed edges of color $c$ form a subset of the edges of $\Gamma_{1,c}$. Hence, with Claim~\ref{clm:rainbow-conn.random.graphs}, we a.a.s. get $\DegrKOut{D-}{M}{c}{v} \leq \Degr_{\Gamma_{1, c}}^{D-}(v) \leq \frac{n}{b}$ for each side $D$, color $c$ and every vertex $v\in V_D$. If $v\notin V_D$, then 
$\DegrKOut{D-}{M}{c}{v}=0$ by definition.

\smallskip

It remains to verify (iii). 
We want to show, that a.a.s. for every side $D$, every color $c$ and every set $A \subset V$ with $\abs{A} \leq b$, 
the inequality $\DegrKOut{D+}{M}{c}{A} > \frac{\abs{A}n}{10^4b}$ holds. 
To do so, we show, that with high probability there is no pair of vertex sets 
$A\subseteq V$ and $B\subseteq V_D$ with $\abs{A} \leq b$ and $\abs{B} = \frac{\abs{A}n}{10^4b}$, 
such that all outgoing edges from the set $A$ of color $c$ into $V_D$ are directed to a vertex in $B$. 
We do that using the union bound over all such pairs. 
Let us start by fixing sets $A$ and $B$
as described.
In total, Maker claims $|A|\cdot \frac{n}{100b}$ outgoing edges from $A$ into $V_D$.
Each time she performs a move of type 1 at a vertex $v\in A$ in layer $c$ corresponding to a box $B_{v, c}^D$ in the MinBox game, she exposes edges at $v\in A$ u.a.r.~from the set of all incident edges towards $V_D$ in layer $c$ that she did not expose at vertex $v$ yet, until she finally obtains an edge not claimed by Breaker. This edge which she then claims is thus chosen u.a.r.~among the at least $0.86 \frac n 2$
edges at $v$ which were neither exposed at $v$ yet nor claimed by Breaker. Therefore, the probability that this one edge goes to a vertex in $B$ is at most $\frac{\abs{B}}{0.86 \frac n 2}$. Hence, the probability that all of Maker's outgoing edges at $A$ into $V_D$ go to $B$ is at most $\left(\frac{\abs{B}}{0.86 \frac n 2}\right)^{\abs{A} \frac{n}{100b}}$.

Using the union bound over all such sets $A\subseteq V$ of size at most $b$ and all corresponding sets $B\subseteq V_D$ of size $\frac{ \abs{A} n}{10^4b}$, we see that (iii) fails with probability at most
\begin{equation*}
    \begin{split}
    &\sum_{a = 1}^b \binom{n}{a} \binom{\frac n 2}{\frac{a n}{10^4b}} 
        \left(\frac{\frac{a n}{10^4b}}{0.86 \frac n 2}\right)^{\frac{an}{100b}} 
    \leq ~ \sum_{a = 1}^b n^a \left(\frac{e \cdot 10^4b}{2 a}\right)^\frac{a n}{10^4b} 
        \left( \frac{a}{4300 b}\right)^\frac{an}{100b} \\
    = &\sum_{a=1}^b \left( n\left(\frac{a}{b}\right)^{\frac{n}{100b} 
        - \frac{n}{10^4b}}\left[ (5000e)^\frac{1}{10^4}\left(\frac{1}{4300}
        \right)^{\frac{1}{100}}\right]^\frac{n}{b}\right)^a 
    \leq ~ \sum_{a = 1}^b \left( n \cdot 1 \cdot 0.93^\frac{n}{b}\right)^a = o(1)
    \end{split}
\end{equation*}
by the choice of $b$.
\end{proof}

With the previous claims in hand, we
can now deduce the following properties 
that we need for analyzing the balancing game.

\begin{claim}\label{clm:rainbow-conn.neighboroods}
    A.a.s.~at the end of the game, the following is true for each $v, w \in V$
    and each $(C_L, c, C_R) \in \Phi$:
\begin{enumerate}
\item[(i)] $\frac{1}{2}(\frac{n}{10^4b})^{k} \leq |\NeighsOutL{M}{C_L}{v}| \leq  (\frac{n}{100b})^{k}$ ~ and ~
    $\frac{1}{2}(\frac{n}{10^4b})^{k} \leq |\NeighsOutR{M}{C_R}{w}| \leq  (\frac{n}{100b})^{k}$,

\item[(ii)]   
$e_{\Gamma_{2,c}}(\NeighsOutL{M}{C_L}{v}, \NeighsOutR{M}{C_R}{w}) \geq \frac{1}{2} p |\NeighsOutL{M}{C_L}{v}| |\NeighsOutR{M}{C_R}{w}|$.
\end{enumerate}    
\end{claim}

\begin{proof}
    We condition on the likely events from 
    Claim~\ref{clm:rainbow-conn.random.graphs}
    and Claim~\ref{clm:rainbow-conn.expansion}.
    We start with the proof of (i), and by symmetry, we only prove the bounds on
   $|\NeighsOutL{M}{C_L}{v}|$. 
    In order to do so, for every $i\in [k]$, we define $C_L(i)$ as the 
    subsequence of $C_L$ consisting of the first $i$ elements only,
    and we let $c_i$ denote the $i$-th color in $C_L$.
   The Claim~\ref{clm:rainbow-conn.expansion}(i) directly implies that
   $|\NeighsOutL{M}{C_L(i)}{v}|  \leq (\frac{n}{100b})^{i}$ for every $i\in [k]$.
   This gives the upper bound in (i).
   Moreover, for every $i\in [k-1]$, this gives
   $$
   |\NeighsOutL{M}{C_L(i)}{v}|  \leq \left(\frac{n}{100b}\right)^{k-1} = o(b)
   $$
   and hence, Claim~\ref{clm:rainbow-conn.expansion}(iii)
   can be applied for every $A\subseteq \NeighsOutL{M}{C_L(i)}{v}$.

   For the lower bound in (i), we now show by induction that
   \begin{equation}\label{eq:expansion.induction}  
   |\NeighsOutL{M}{C_L(i)}{v}|  \geq \left(\frac{n}{10^4b}\right)^{i} \left(1-\frac{i 10^{4k} b}{n}\right)
   \end{equation}
   holds for every $i\in [k]$. Note that then the lower bound on
   $|\NeighsOutL{M}{C_L}{v}|$
   in (i) follows, thus finishing the proof of property (i).
   The inequality \eqref{eq:expansion.induction} is surely true for $i=1$
   because of Claim~\ref{clm:rainbow-conn.expansion}(i).
   For the induction step assume that 
   $$|\NeighsOutL{M}{C_L(i-1)}{v}|  \geq \left(\frac{n}{10^4b}\right)^{i-1} \left(1-\frac{ (i-1) 10^{4k} b}{n}\right)$$
   is true. Then recall that, by definition, for every 
   vertex $x\in \NeighsOutL{M}{C_L(i-1)}{v}$ there exists a directed rainbow path from
   $v$ to $x$ with color sequence $C_L(i-1)$. Each vertex belonging to
   at least one of such paths is contained in $U_{i-1}:=\bigcup_{j\leq i-1} \NeighsOutL{M}{C_L(j)}{v}$
   and we have
   $$|U_{i-1}| \leq \sum_{j\leq i-1} |\NeighsOutL{M}{C_L(j)}{v}| \leq 2 \left(\frac{n}{100b}\right)^{i-1}
	< \left(\frac{n}{b}\right)^{i-1} .$$
   Moreover, note that $
   \NeighsOutL{M}{c_i}{\NeighsOutL{M}{C_L(i-1)}{v}}\setminus U_{i-1} \subseteq \NeighsOutL{M}{C_L(i)}{v}$.
   By using Claim~\ref{clm:rainbow-conn.expansion}(iii), we get
   $$
   \big|\NeighsOutL{M}{c_i}{\NeighsOutL{M}{C_L(i-1)}{v}}| \geq \frac{n}{10^4b}|\NeighsOutL{M}{C_L(i-1)}{v} \big|
	\geq \left(\frac{n}{10^4b}\right)^{i} \left(1-\frac{(i-1) 10^{4k} b}{n}\right)
   $$
  and therefore,
  \begin{align*}
  |\NeighsOutL{M}{C_L(i)}{v}| 
  & \geq 
  \big|\NeighsOutL{M}{c_i}{\NeighsOutL{M}{C_L(i-1)}{v}}| - |U_{i-1}|  \\
  & \geq 
  \left(\frac{n}{10^4b}\right)^{i} \left(1-\frac{(i-1) 10^{4k} b}{n}\right) - 
   \left(\frac{n}{b}\right)^{i-1} 
  \geq 
  \left(\frac{n}{10^4b}\right)^{i} 
  \left(1-\frac{i 10^{4k} b}{n}\right) .
  \end{align*}
  This completes the inductive proof.

  \smallskip

    Finally, property (ii) of this claim 
    follows directly from property (iii) 
    in Claim~\ref{clm:rainbow-conn.random.graphs}, applied
    with $A:= \NeighsOutL{M}{C_L}{v}$ and 
    $B:= \NeighsOutR{M}{C_R}{w}$. 
\end{proof}

\begin{claim}\label{clm:rainbow-conn.degree.bounds}
    A.a.s.~the moves of type 2 ensure $d_{B \setminus \text{Sp},c}(v) < \delta n$
    for each layer $c$ and each vertex $v$.
\end{claim}

\begin{proof}
Let us condition on the good events from
Claim~\ref{clm:rainbow-conn.random.graphs}.
Then, $d_{\Gamma_2,c}(v) \leq (1 + \delta) n p$ and hence also $d_{M \cap \Gamma_2,c}(v) \leq (1 + \delta) n p$. 
Whenever Maker claims an edge at a vertex $v$ according to a move of type 2 (which then belongs to 
$\Gamma_{2,c}$), this edge is also added as a Maker's edge to the SBG. Hence, the number of Maker's elements in a hyperedge $F_v^c$ 
of this simulated game is upper bounded by 
$d_{M \cap \Gamma_2,c}(v)$.
Moreover, by the way the SBG is updated, the number of Breaker's elements in the same 
hyperedge $F_v^c$ equals 
$d_{B \setminus \text{Sp},c}(v)$.
Since Maker wins the SBG 
by Claim~\ref{clm:rainbow-conn.SBGcheck}, we obtain that 
$$d_{B \setminus \text{Sp},c}(v) \leq (3b + 1) (d_{M \cap \Gamma_2,c}(v) + \ell_2)
< 4b \cdot 2np \leq \delta n,
$$
by the choice of $p,b$. 
\end{proof}

Finally, we reach our key claim that ensures that Breaker a.a.s.~does not get too many relevant edges between
any sets $\NeighsOutL{M}{C_L}{v}
$ and $\NeighsOutR{M}{C_R}{w}
$
with $v, w \in V$ and $\mathcal{C} = (C_L, c, C_R) \in \Phi$.

\begin{claim}\label{clm:rainbow-conn.between.neighborhoods}
    A.a.s.~for all $(C_L, c, C_R) \in \Phi$ and for all $v, w \in V$ we have
   $$ e_{B\setminus \text{Sp},c}(\NeighsOutL{M}{C_L}{v}
,  \NeighsOutR{M}{C_R}{w}
 ) \leq \delta^{0.1}\left(\frac{n}{b}\right)^{2k} $$
   at the end of the game.
\end{claim}

\begin{proof}
In the following, let us condition on all the likely events from all the previous claims.

For fixed $v,w\in V$ and $(C_L,c,C_R)\in \Phi$, 
there are three kinds of ways how a Breaker edge $xy$
outside of the set $\text{Sp}$
can become an edge between $\NeighsOutL{M}{C_L}{v}
$ and $\NeighsOutR{M}{C_R}{w}
$ in layer $c$:
\begin{enumerate}
    \item[(i)] $x\in \NeighsOutL{M}{C_L}{v}
$ and 
    $y\in \NeighsOutR{M}{C_R}{w}
$ before Breaker claims $xy$, but Breaker claims $xy$ before $xy$ is added to $B_c(v,w,C_L,C_R)$
    in the Spooky Balancing Game simulated in moves of type 3,
    \item[(ii)] $x\in \NeighsOutL{M}{C_L}{v}
$ and 
    $y\in \NeighsOutR{M}{C_R}{w}
$ before Breaker claims $xy$, but Breaker claims $xy$ after $xy$ is added to $B_c(v,w,C_L,C_R)$
    in the Spooky Balancing Game simulated in moves of type 3,
    \item[(iii)] Breaker claims $xy$ before $x$ or $y$ is added to $\NeighsOutL{M}{C_L}{v}
$ or
    $\NeighsOutR{M}{C_R}{w}
$.
\end{enumerate}

In the following, we show that
for fixed $v,w\in V$ and $(C_L,c,C_R)\in \Phi$,
with probability at least $1-\exp(-\Theta(n/b))$,
each of the cases (i)--(iii) contributes at most
$\frac{1}{3}\delta^{0.1}\left(\frac{n}{b}\right)^{2k}$
to the number $e_{B\setminus \text{Sp},c}(\NeighsOutL{M}{C_L}{v}
,  \NeighsOutR{M}{C_R}{w}
 )$.
Then, Claim~\ref{clm:rainbow-conn.between.neighborhoods}
follows by a union bound over all choices of 
$v,w\in V$ and $(C_L,c,C_R)\in \Phi$.

\smallskip

Let us fix $v,w\in V$ and $(C_L,c,C_R)\in \Phi$
from now on.
We start to estimate the number of the edges in (i). 
For this, note that an edge $xy$ can only be counted in (i) if 
$x$ or $y$ is added to $\NeighsOutL{M}{C_L}{v}
$ or
    $\NeighsOutR{M}{C_R}{w}
$ by a move of type 1 and
    then Breaker claims $xy$ before Maker does the next update in the SBG simulated for moves of type 3. Also note that Breaker can only claim 2b edges between consecutive moves of type 1 and 3, and case (i) can only happen each time when $\NeighsOutL{M}{C_L}{v}
$ or
    $\NeighsOutR{M}{C_R}{w}
$ increases, i.e., at most
$2\left( \frac{n}{100b} \right)^{k}$ times
by Claim~\ref{clm:rainbow-conn.neighboroods}.
     Hence,
    there can be at most
    \begin{align*}
        2b\cdot 2\left( \frac{n}{100b} \right)^{k} < \frac{1}{3}\delta^{0.1}\left(\frac{n}{b}\right)^{2k} 
    \end{align*}
    edges counted in (i) until the end of the game.

\smallskip

As next, we consider the edges in (ii).
The edges in this case are precisely the 
edges claimed by Breaker in the set $B_c(v,w,C_1,C_2)$ in the simulated SBG in moves of type 3;
let their number be $b_c(v,w)$.
Moreover, the edges claimed by Maker 
in the set $B_c(v,w,C_1,C_2)$
are edges in
$E_{\Gamma_{2,c}\cap M}(\NeighsOutL{M}{C_L}{v}
,\NeighsOutR{M}{C_R}{w}
)$;
let their number be $m_c(v,w)$.
By Claim~\ref{clm:rainbow-conn.random.graphs}
we have
$$
m_c(v,w)
\leq 
|E_{\Gamma_{2,c}}(\NeighsOutL{M}{C_L}{v}
,\NeighsOutR{M}{C_R}{w}
)|
\leq
2p |\NeighsOutL{M}{C_L}{v}||
\NeighsOutR{M}{C_R}{w}| ,
$$
and, since Maker wins the SBG by Claim~\ref{clm:rainbow-conn.SBGcheck},
we have
$$
m_c(v,w) \geq \frac{1}{3b+1} \cdot (m_c(v,w) + b_c(v,w)) - \ell_3 .
$$ 
Thus, the number of edges in (ii) can be bounded by
\begin{align*}
b_c(v,w) \leq 3b\cdot m_c(v,w) + (3b+1) \cdot \ell_3 
 \leq 6bp \left( \frac{n}{100b} \right)^{2k} + 4\delta b p(pn)^{2k} 
 < \frac{1}{3}\delta^{0.1} \left( \frac{n}{100b} \right)^{2k} ,
\end{align*}
where the second inequality holds by Claim~\ref{clm:rainbow-conn.neighboroods} and the definition of $\ell_3$, and
the last inequality holds by the choice of $b,p$ and $\delta$.

\smallskip

Finally, we consider the edges counted in (iii),
which turns out to be the trickiest part of this proof.
Let us call a vertex $x$ a \emph{bad vertex 
of the first kind} if at some point during the game
$x\notin \NeighsOutL{M}{C_L}{v}
$ and
$e_{B\setminus \text{Sp},c}(x,\NeighsOutR{M}{C_R}{w}
)\geq \delta^{0.5} \left( \frac{n}{b} \right)^{k}$ hold.
Similarly, let us call a vertex
$x$ a \emph{bad vertex 
of the second kind} 
if at some point during the game
$x\notin \NeighsOutR{M}{C_R}{w}
$ and
$e_{B\setminus \text{Sp},c}(x,\NeighsOutL{M}{C_L}{v}
)\geq \delta^{0.5} \left( \frac{n}{b} \right)^{k}$ hold.
That is, these vertices, when added to 
$\NeighsOutL{M}{C_L}{v}
$ or $\NeighsOutR{M}{C_R}{w}
$, can immediately add 
a number of Breaker edges between $\NeighsOutL{M}{C_L}{v}
$ and $\NeighsOutR{M}{C_R}{w}
$ that is proportional to 
the final size of these sets.
We claim that
with probability at least $1-\exp(-\Theta(n/b))$
the number of bad vertices of the first kind
that are added to $\NeighsOutL{M}{C_L}{v}$ 
and the number of bad vertices of the second kind
that are added to $\NeighsOutR{M}{C_R}{w}$ 
can be bounded by 
$\delta^{0.2} \left( \frac{n}{b} \right)^{k}$.
Before proving this, let us explain how 
our desired bound on the number of edges in (iii) follows from that.

\smallskip

An edge $xy$ as described in (iii) can only be added
between $\NeighsOutL{M}{C_L}{v}
$ and $\NeighsOutR{M}{C_R}{w}
$
through moves of type 1,
when $x$ is added to $\NeighsOutL{M}{C_L}{v}
$ or
$y$ is added to $\NeighsOutR{M}{C_R}{w}$ after Breaker claimed $xy$.
In total, among the vertices $x$ that are
added to $\NeighsOutL{M}{C_L}{v}
$,
there are at most 
$\delta^{0.2} \left( \frac{n}{b} \right)^{k}$
vertices that are bad of the first kind,
and each of them cannot contribute more than
$|\NeighsOutR{M}{C_R}{w}
| \leq \left(\frac{n}{b}\right)^{k}$
to the edges counted in (iii),
while all the other vertices added to
$\NeighsOutL{M}{C_L}{v}
$ contribute at most 
$\delta^{0.5}\left(\frac{n}{b}\right)^{k}$
edges counted in (iii).
By symmetry, the same can be said about vertices added to
$\NeighsOutR{M}{C_R}{w}
$, but with bad vertices of the second kind. Therefore, the number of edges 
counted in (iii)
can be bounded from above by
\begin{align*}
& \sum_{j\in [2]}
\left(
\delta^{0.2} \left( \frac{n}{b} \right)^{k}\cdot \left( \frac{n}{b} \right)^{k} +  
\left(\frac{n}{b}\right)^{k} \cdot 
\delta^{0.5} \left(\frac{n}{b}\right)^{k} \right) <
\frac{1}{3} \delta^{0.1}\left(\frac{n}{b}\right)^{2k} .
\end{align*}

\smallskip

Hence, it remains to show that 
with probability at least $1-\exp(-\Theta(n/b))$
the number of bad vertices 
that are added to $\NeighsOutL{M}{C_L}{v}
$
and $\NeighsOutR{M}{C_R}{w}
$
can be bounded as mentioned above.
By symmetry we will only give a full argument for the bad vertices of the first kind, which we simply call \emph{bad} from now on. 
Moreover, for the sequence of colors $C_L=(c_1,\ldots,c_{k})$
let us denote 
$C_L(i)=(c_1,\ldots,c_{i})$ and
$C_L(>i)=(c_{i+1},\ldots,c_{k})$.

Since $|\NeighsOutR{M}{C_R}{w}
|\leq \left( \frac{n}{b} \right)^{k}$
and $d_{B\setminus \text{Sp},c}(u)\leq \delta n$ for every $u\in V$
hold throughout the game, there can be at most
$\delta n \cdot \left( \frac{n}{b} \right)^{k}$
edges in $B\setminus \text{Sp}$ incident with $\NeighsOutR{M}{C_R}{w}$ throughout the game. 
In particular, the number of bad vertices
is bounded by
$$\frac{\left( \frac{n}{b} \right)^{k} \cdot \delta n}{\delta^{0.5} \left( \frac{n}{b} \right)^{k}} \leq \delta^{0.5}n .$$
From this bound we can already conclude the following: whenever Maker exposes and claims an edge at a vertex $z\in \NeighsOutL{M}{C_L(k-1)}{v}
$ in layer $c_{k}$ into $V_L$, which happens at most $\left(\frac{n}{100b} \right)^{k}$ times, the probability of taking an edge $zx$ such that $x$ is bad
is bounded from above by 
$$
\frac{\delta^{0.5}n}{0.86\frac{n}{2}}<4\delta^{0.5},
$$
as there are at least $0.86\frac{n}{2}$ free edges at $z$ in layer $c_k$ into $V_L$ at that moment, by Claim~\ref{clm:rainbow-conn.Breaker-degree}. By applying a Chernoff argument (Lemma~\ref{lemma:Chernoff_modified}) it holds with probability at least $1-\exp(-\Theta( (n/b)^{k} ))$
that Maker claims at most
$$
2\cdot 4\delta^{0.5}\cdot \left( \frac{n}{100b} \right)^{k}
=
8\delta^{0.5}\left( \frac{n}{100b} \right)^{k}
$$
edges $zx$ in layer $c_{k}$ 
with
$z\in \NeighsOutL{M}{C_L(k-1)}{v}
$ and such that
$x\in V_L$ is a bad vertex that then gets added to $\NeighsOutL{M}{C_L}{v}$. From now on, we will condition on this likely event.

\smallskip

In order to bound the number of all other bad vertices that are added to $\NeighsOutL{M}{C_L}{v}$, let us call a vertex $y$ \emph{crazy of level $i$}, with $i\in [k-1]$,
if during the game there appear at least 
$\delta^{0.25} \left( \frac{n}{b} \right)^{k-i}$ 
bad vertices $x$ such that
$x\in \NeighsOutL{M}{C_L(>i)}{y}$. Fix $i\in [k-1]$.
By Claim~\ref{clm:rainbow-conn.expansion}(ii), 
for every vertex $x$ there can at most
$\left( \frac{n}{b} \right)^{k-i}$ vertices $z$ such that
$x\in \NeighsOutL{M}{C_L(>i)}{z}$, and thus, the number of
crazy vertices of level $i$ can be bounded by
$$\frac{\left( \frac{n}{b} \right)^{k-i} \cdot \delta^{0.5} n}{\delta^{0.25} \left( \frac{n}{b} \right)^{k-i}} \leq \delta^{0.25}n .$$

Now, consider any round in layer $c_i$ where $c_i$ is the $i$-th color in the sequence $C_L$. Each time an edge $zy$ of this layer is exposed at a vertex $z\in \NeighsOutL{M}{C_L(i-1)}{v}$,
there are at least $0.86\frac{n}{2}$ free edges at $z$ in layer $c_i$ into $V_L$.
Hence, the probability of taking $zy$ in such a way that $y\in V_L$ is crazy of level $i$ is upper bounded by
$$
\frac{\delta^{0.25}n}{0.86\frac{n}{2}} < 4\delta^{0.25}.
$$
Now, note that 
$|\NeighsOutL{M}{C_L(i-1)}{v}
| \leq \left( \frac{n}{100b} \right)^{i-1}$
holds throughout the game
by Claim~\ref{clm:rainbow-conn.expansion},
and at each vertex we expose 
$2\left( \frac{n}{100b} \right)$ edges for $\Gamma_{1,c_i}$ towards $V_L$ in moves of type 1.
Hence, there are at most $2\left( \frac{n}{100b} \right)^{i}$ rounds (of type 1) in which Maker exposes edges of $\Gamma_{1,c_i}$ at a vertex $z\in \NeighsOutL{M}{C_L(i-1)}{v}$.
By a Chernoff argument (Lemma~\ref{lemma:Chernoff_modified}) it holds with probability at least $1-\exp(-\Theta( (n/b)^i ))$
that Maker claims at most
$$
2\cdot 4\delta^{0.25}\cdot 2\left( \frac{n}{100b} \right)^{i}
=
16\delta^{0.25}\left( \frac{n}{100b} \right)^{i}
$$
directed edges $(z,y)$ in layer $c_i$ in moves of type 1 such that $z\in \NeighsOutL{M}{C_L(i-1)}{v}$
and $y\in V_L$ is crazy of level $i$. After doing a union bound over all $i\in [k-1]$, we 
see that this holds for all $i\in [k-1]$ 
with probability at least
$1-\exp(-\Theta (n/b) )$.
We condition on this to hold from now on.

\smallskip

Then, note that a bad vertex $x\in V_L$  
can only be added to $\NeighsOutL{M}{C_L}{v}
$ by a move of type 1,
in which either, for some $i\in [k-1]$, Maker claims a directed edge $(z,y)$ in layer $c_i$ such that
$z\in \NeighsOutL{M}{C_L(i-1)}{v}
$ and 
$x\in \NeighsOutL{M}{C_L(>i)}{y}
$, or
Maker claims a directed edge $(y,x)$ 
with $y\in \NeighsOutL{M}{C_L(k-1)}{v}
$ and $x$ being bad. 
For the second case, we already know that
this happens at most 
$8\delta^{0.5}\left( \frac{n}{100b} \right)^{k}$
times. So, consider any $i\in [k-1]$. There are at most
$\left( \frac{n}{100b} \right)^{i}$ rounds in which Maker 
claims an edge which was exposed at a vertex 
$z\in \NeighsOutL{M}{C_L(i-1)}{v}
$ towards $V_L$. Among these rounds, there are at most  $16\delta^{0.25}\left( \frac{n}{100b} \right)^{i}$
rounds in which the exposed edge $zy$ leads to a crazy
vertex $y\in V_L$ of level $i$. By Claim~\ref{clm:rainbow-conn.expansion}, each of these rounds can add at most
$|\NeighsOutL{M}{C_L(>i)}{y}
|\leq \left( \frac{n}{100b} \right)^{k-i}$
bad vertices to $\NeighsOutL{M}{C_L}{v}$.
In each of the remaining rounds without a crazy vertex $y$ of level $i$, at most 
$\delta^{0.25} \left( \frac{n}{b} \right)^{k-i}$
bad vertices can be added to $\NeighsOutL{M}{C_L}{v}$.
Therefore, summing up everything, we conclude that
the number of bad vertices 
added to $\NeighsOutL{M}{C_L}{v}
$ 
can be 
upper bounded by
\begin{align*}
 & 8\delta^{0.5} \left( \frac{n}{100b} \right)^{k} + 
\sum_{i\in [k-1]} \left(
16\delta^{0.25}\left( \frac{n}{100b} \right)^{i}\cdot 
\left( \frac{n}{100b} \right)^{k-i}
+
\left( \frac{n}{100b} \right)^{i}\cdot 
\delta^{0.25} \left( \frac{n}{b} \right)^{k-i}
\right) \\
< ~ &
 \delta^{0.5} \left( \frac{n}{b} \right)^{k} + 
\sum_{i\in [k-1]} 
\delta^{0.25} \left(\frac{n}{b} \right)^{k}
<
\frac{1}{3} \delta^{0.1} \left( \frac{n}{b} \right)^{k} .
\end{align*}
This finishes the proof of the claim.
\end{proof}

With our final claim, we get closer to our goal of showing that Maker gets rainbow paths as claimed in Theorem~\ref{thm:rainbow-conn.s.constant}(b).

\begin{claim}\label{clm:rainbow-conn.many.paths}
A.a.s.~for every $v,w,x\in V$ with $x\neq v,w$ and every sequence $C\in \Phi$, Maker 
creates at least $\frac{1}{5} p |\NeighsOutL{M}{C_L}{v}
||\NeighsOutR{M}{C_R}{w}|$
rainbow paths between $v$ and $w$
that do not contain $x$
and such that the edges are colored according to the sequence $C$.
\end{claim}

\begin{proof}
We condition on the good events of all previous claims. 
First note that the edges counted in Claim~\ref{clm:rainbow-conn.between.neighborhoods} are precisely
the edges of Breaker between
$\NeighsOutL{M}{C_L}{v}
$ and $\NeighsOutR{M}{C_R}{w}
$ in layer $c$ that could belong to
$\Gamma_{2,c}$ at the end of the game
(because all edges belonging to $B\cap \text{Sp}$ will not be part of 
$\Gamma_{2,c}$). Now each of these edges 
counted in Claim~\ref{clm:rainbow-conn.between.neighborhoods} is exposed for $\Gamma_{2,c}$ at the end of the game;
each edge is added to $\Gamma_{2,c}$
with probability $p$.
Applying another Chernoff argument (Lemma~\ref{lemma:Chernoff_modified}) together with a union bound, it follows that a.a.s.~for all $(C_L, c, C_R) \in \Phi$ and for all $v, w \in V$ we have
   $$ e_{B\cap \Gamma_{2,c}}(\NeighsOutL{M}{C_L}{v}
,  \NeighsOutR{M}{C_R}{w}
 ) \leq 2p \cdot \delta^{0.1}\left(\frac{n}{b}\right)^{2k} 
< \frac{1}{10} p |\NeighsOutL{M}{C_L}{v}
||\NeighsOutR{M}{C_R}{w}
| , $$
where the last inequality uses 
property (i) from Claim~\ref{clm:rainbow-conn.neighboroods} and the choice of $\delta$.
Since, at the end of the game,
each edge of $\Gamma_{2,c}$ belongs to either Maker or Breaker, using Claim~\ref{clm:rainbow-conn.neighboroods}(ii)
we conclude that
\begin{align*}
    e_{M\cap \Gamma_{2,c}}(\NeighsOutL{M}{C_L}{v},  \NeighsOutR{M}{C_R}{w} ) 
   & =
   e_{\Gamma_{2,c}}(\NeighsOutL{M}{C_L}{v},  \NeighsOutR{M}{C_R}{w} ) -
   e_{B\cap \Gamma_{2,c}}(\NeighsOutL{M}{C_L}{v},  \NeighsOutR{M}{C_R}{w} ) \\
   & 
    > \frac{2}{5} p |\NeighsOutL{M}{C_L}{v}||\NeighsOutR{M}{C_R}{w}|.
\end{align*}

Now, fix any $v,w,x\in V$ and $C = (C_L,c,C_R)\in \Phi$ as described in the statement of the claim. W.l.o.g.~let $x\in V_L$.
For all vertices $z\in \NeighsOutL{M}{C_L}{v}$
there exists a directed rainbow path $P_z$ in Maker's graph starting in $v$ and ending in $z$
whose edges are colored according to the sequence $C_L$. Set $\operatorname{N}_{bad}^+\subseteq \NeighsOutL{M}{C_L}{v}$ to be the subset of these vertices $z$ for which $P_z$ contains the given vertex $x$. Then
$z\in \operatorname{N}_{bad}^+$ implies
$z\in \NeighsOutL{M}{C_L(>i)}{x}$ for some $i\in [k-1]$ or $z=x$. Hence, by Claim~\ref{clm:rainbow-conn.expansion}(i) we have 
$$
|\operatorname{N}_{bad}^+|
\leq
1+ \sum_{i\in [k-1]}
|\NeighsOutL{M}{C_L(>i)}{x}|
\leq
1+ \sum_{i\in [k-1]}
\left( \frac{n}{100b} \right)^{k-i}
<
\frac{1}{10^3}\left( \frac{n}{10^4b} \right)^k,
$$
where the last expression is chosen so that we can apply Claim~\ref{clm:rainbow-conn.random.graphs}(iii).
Using this claim 
and Claim~\ref{clm:rainbow-conn.neighboroods}(i) we conclude
$$
e_{\Gamma_{2,c}}(\operatorname{N}_{bad}^+,\NeighsOutR{M}{C_R}{w}
) \leq 2p\cdot \frac{1}{10^3}\left( \frac{n}{10^4b} \right)^k
 \cdot |\NeighsOutR{M}{C_R}{w}|
\leq
\frac{1}{100} p\cdot |\NeighsOutL{M}{C_L}{v}||\NeighsOutR{M}{C_R}{w}| .
$$
In particular, we conclude
\begin{align*}
    e_{M\cap \Gamma_{2,c}}(\NeighsOutL{M}{C_L}{v}\setminus \operatorname{N}^+_{bad},  \NeighsOutR{M}{C_R}{w} ) 
   & \geq
    e_{M\cap \Gamma_{2,c}}(\NeighsOutL{M}{C_L}{v},  \NeighsOutR{M}{C_R}{w} )    
 -
 e_{\Gamma_{2,c}}(\operatorname{N}_{bad}^+,\NeighsOutR{M}{C_R}{w}
)   \\ 
& > \frac{1}{5} p |\NeighsOutL{M}{C_L}{v}||\NeighsOutR{M}{C_R}{w}|.
\end{align*}
Now note that every edge $yz\in E_{M\cap \Gamma_{2,c}}(\NeighsOutL{M}{C_L}{v}\setminus \operatorname{N}^+_{bad},  \NeighsOutR{M}{C_R}{w} )$ is contained in at least one rainbow path between $v$ and $w$ in Maker's graph
that does not contain $x$
and which starts with the colors of $C_L$,
then uses color $c$ on the edge $yz$,
and then continues with colors of $C_R$ in reversed order. This proves the claim.
\end{proof}

Based on this last claim, we can finish the proof of Theorem~\ref{thm:rainbow-conn.s.constant}(b). We start with the case when $s$ is odd and
hence $s=2k+1$. Then the above claim already implies that for every $v,w\in V$ and every sequence $C\in \Phi = S^s$, Maker has at least
$$\frac{1}{5} p |\NeighsOutL{M}{C_L}{v}
||\NeighsOutR{M}{C_R}{w}|
= \Theta(n^{2k}b^{-(2k+1)}) = \Theta(n^{s-1}b^{-s})
$$
rainbow paths between $v$ and $w$
with colors according to the sequence $C$.
This is what we wanted to prove.

\smallskip

Hence, let $s$ be even from now on, and thus $s=2k+2$. Let $C'=(c_0,c_1,\ldots,c_{2k+1})\in S^s$ be any sequence of colors. 
Write $C=(c_1,\ldots,c_{2k+1})$. We know that
$\DegrKOut{L+}{M}{c_0}{v}=\frac{n}{100b}$.
For each $v'\in \NeighsOutL{M}{c_0}{v}$
we can apply Claim~\ref{clm:rainbow-conn.many.paths} to ensure
$$\frac{1}{5} p |\NeighsOutL{M}{C_L}{v'}
||\NeighsOutR{M}{C_R}{w}|
= \Theta(n^{2k}b^{-(2k+1)})
$$
rainbow paths between $v'$ and $w$
with colors according to the sequence $C$
such that these paths do not contain $v$.
Extending by the edge $vv'$ in color $c_0$,
and running over all possible
$v'\in \NeighsOutL{M}{c_0}{v}$, we see that Maker claims at least
$$\frac{n}{100b} \cdot \Theta(n^{2k}b^{-(2k+1)})
= \Theta(n^{2k+1}b^{-(2k+2)}) = \Theta(n^{s-1}b^{-s})
$$
rainbow paths between $v$ and $w$ 
with colors according to the sequence $C'$,
as required.
\end{proof}

\subsection{Proof of Theorem~\ref{thm:diameter.new}}

\begin{proof}[Proof of Theorem~\ref{thm:diameter.new}]
    For a Breaker's strategy, note that the proof of Theorem~\ref{thm:rainbow-conn.s.constant}(a) shows that
    Breaker can block all paths of length at most $s$ between a fixed pair of vertices $v,w$ if Maker and Breaker play a $(1:b)$ game on a complete multigraph on $n$ vertices with edge multiplicity $s$,
    if $b\geq 24s n^{1-1/\lceil s/2\rceil}$.
    The same strategy can be applied in the diameter game $\mathcal{D}_{s,n}$. Hence,
    $b_{\mathcal{D}_{s,n}}\leq 24s n^{1-1/\lceil s/2\rceil}$.
    For a Maker's strategy, note that in the proof of Theorem~\ref{thm:rainbow-conn.s.constant}(b), we can choose
    $C\in S^s$ to consist of $s$ times the same color $c$. Then Maker obtains
    paths of length $s$ that are monochromatic in color $c$ between each pair of vertices $v,w$. This strategy can be applied 
    to a game on $K_n$ by identifying $K_n$ with layer $c$ from above (and by imagining further layers that Maker does not play on). Hence, it provides a winning strategy for the diameter game $\mathcal{D}_{s,n}$ when
    $b\leq \gamma n^{1-/\lceil s/2\rceil}$,
    where $\gamma$ is chosen as in the proof of 
    Theorem~\ref{thm:rainbow-conn.s.constant}(b). Therefore,
    $b_{\mathcal{D}_{s,n}}\geq \gamma n^{1-1/\lceil s/2\rceil}$.
\end{proof}

\subsection{Proof of Theorem~\ref{thm:rainbow-conn.s.large}}

\begin{proof}[Proof of Theorem~\ref{thm:rainbow-conn.s.large}]
Let $S$ be a set of $s$ colors,
denote the layers with $G_c$ where $c\in S$
and let $G=\bigcup_{c\in S}G_c$.

\smallskip

\textbf{Upper bound:} 
We start with the upper bound, i.e., Breaker's strategy.
Let $b>(1+\eps)\frac{sn}{\log(n)}$ for an arbitrary
small constant $\eps>0$. Then Breaker can isolate a vertex, i.e., make sure that Maker gets no edge at this vertex. The strategy follows an argument of Chv\'atal and Erd\H{o}s~\cite{chvatal1978biased}: At first, Breaker can play so that after each round $r \leq \frac{b}{4s}$, there is a clique
$A_r$ of size exactly $r$ such that Breaker claims all edges inside $A_r$ in all colors, and Maker has no edge incident with $A_r$ in any color. Indeed, having such a set $A_{r-1}$ after some round $r-1$, Maker can only touch one of its vertices in the next round, reducing the clique by at most one vertex. But then
Breaker can add to his clique up to two vertices $x,y$ that were not touched by Maker yet, by claiming all edges between $x$ and $y$ and $A_{r-1}$. 
Secondly, having reached round $\frac{b}{4s}$, Breaker can isolate one of the vertices of his clique by a Box game argument as follows. For each vertex $v$ in the clique, create a box $B_v$ of all edges incident with $v$ that are still free. Then this leads to $\frac{b}{4s}$ boxes of size $sn(1-o(1))$, and hence, by Theorem~\ref{thm:boxgame}, Breaker (in the role of Maker) can occupy a box $B_v$ completely. That is, he makes sure that Maker gets no edge at the vertex $v$. 

\medskip

\textbf{Lower bound:} 
Let us turn to a suitable Maker's strategy. 
This time, let $b = \gamma \frac{sn}{\log(n)}$ for an arbitrary constant $0<\gamma<\frac{1}{2}$. 
Fix a constant $\delta>0$ such that
$\gamma<\frac{1}{2}-\delta$, and let $\eps'>0$
be given by Lemma~\ref{lem:degree.game.variant}
with input $\delta$, and $\eps\in (0,\eps')$.
Moreover let 
$$t:=\lceil \log^{(2)}(n) \rceil
~~ \text{and} ~~
b':=\left\lceil \frac{t}{t-1} \cdot b \right\rceil .$$
Our goal is to show that Maker can create rainbow paths between every pair of vertices. For this, we partition 
the set $S$ into three sets
$S_1,S_2,S_3$ such that 
$$|S_3|=\frac{s}{\log^{(2)}(n)} ~~
\text{and} ~~
s':=|S_1|=|S_2|=(1-o(1))\frac{s}{2} .$$ 
Let $G_1$ and $G_2$ be the submultigraphs of $G$ 
induced by the edges with colors in $S_1$ and
$S_2$, respectively. Let $H_i$ be a copy of $G_i$,
with $i\in [2]$,
such that $H_1$ and $H_2$ are vertex-disjoint.
The copy of a vertex $v\in V(G)$ in the graph $H_i$
will always be denoted as $v_i$.
Moreover, we set $H=(V(H_1)\cup V(H_2),E(H_1)\cup E(H_2))$.
As in previous proofs,
we let 
$N^+_{M,c}(v)$ denote the set of vertices
$w$ such that Maker has an oriented edge from
$v$ to $w$ in color $c$,
and we set
$d^+_{M,c}(v)=|N^+_{M,c}(v)|$.
We also let 
$d_{B,c}(v)$ denote the number of Breaker's edges at $v$ 
in color $c$, and we define
$d_{M,S_i}(v) := \sum_{c\in S_i} d^+_{M,c}(v)$ and
$d_{B,S_i}(v) := \sum_{c\in S_i} d_{B,c}(v)$
for $i\in [2]$.

\medskip

\textbf{Strategy:} Maker follows two types of moves,
where moves of type $2$ happen only every $t$ rounds.
If at any point in the game Maker is not able to follow the given instructions, she forfeits the game.

\medskip

\underline{Type 1:}
In parallel, Maker simulates 
$\text{MinDeg}^+(H,\frac{\eps s'n}{b'},b',\eps)$.

\smallskip

In the actual game, Maker always considers
sequences of $t$ consecutive rounds 
starting with a round
$1+t\cdot i$ and ending with a round $t(i+1)$ 
for some $i\in \mathbb{N}_0$.

Fix $i\in\mathbb{N}_0$. In the round $1+t\cdot i$, Maker does not make a move of type 1. Let $B_{1}$ be the set of Breaker's edges claimed in that round. Moreover, let
$B_{1}=B_{1,2}\cup B_{1,3}\cup \ldots \cup B_{1,t}$
be an equipartition of $B_1$.
Then Maker plays the next rounds
$j+t\cdot i$ with $j\in \{2,\ldots,t\}$ as follows:
\begin{enumerate}
\item[(1)] Let $B_j$ be the set of edges Breaker claims in round
$j+t\cdot i$. Then Maker sets $B_j' := B_{1,j} \cup B_j$.
She updates $\text{MinDeg}^+$ by letting Breaker claim all the edges in $B_j'$ within one move.
\item[(2)] Then strategy $\mathcal{S}^+$ tells Maker 
at which vertex $v_i\in V(H_i)$ to choose and direct an arbitrary edge. 
\item[(3)] Having $v\in V(G)$ and $i\in [2]$ fixed this way, in the actual game 
Maker does the following:
she takes a color $c\in S_i$ such that Maker has
no outgoing edge at $v$ with color in $S_i$ yet, and
such that $d_{B,c}(v)$ is smallest among all such choices.
She then claims an edge at $v$ 
u.a.r.~among all non-Breaker edges that are incident with $v$ in color $c$, and she orients the edge such that it is outgoing at $v$.
\item[(4)] In the simulated game $\text{MinDeg}^+$, Maker claims the copy of that edge in $H_i$ and also orients it to be outgoing at $v_i$.
\end{enumerate}

\medskip

\underline{Type 2:} 
Moves of type 2 are played only in every round $r\equiv 1$ (mod $t$). Here Maker ensures to get an edge between any two disjoint sets of size $\frac{n}{\log^{(2)}(n)}$
of some color from $S_3$. Details on how Maker can achieve this are given later in the strategy discussion.

\bigskip

\textbf{Strategy discussion:}
We start by analysing the moves of type 1.

\begin{claim}\label{lem:MinDeg.legal}
As long as Maker does not forfeit the game,
in the simulated game $\text{MinDeg}^+$ the following holds:
\begin{enumerate}
\item[(a)] The updates in (1) and (4) for $\text{MinDeg}^+$ are legal moves for Maker and Breaker.
\item[(b)] For every $v_i\in V(H)$, as long as $d_{M\cap H}(v_i)\leq \frac{\eps s'n}{b'}$
 we have $d_{B\cap H}(v_i)<(1-\frac{\eps}{2})\frac{sn}{2}$.
\end{enumerate}
\end{claim}

\begin{proof}
We start with (a). Whenever Makes does an update in (1), Breaker obtains
a set $B_j'$ of edges in the game $\text{MinDeg}^+$.
By its definition, the set has size
$|B_j'|=|B_{1,j}|+|B_j|\leq \left\lceil \frac{b}{t-1} \right\rceil + b = b'$. Hence, Breaker's simulated move
is legal. Moreover, every time Breaker receives edges in 
$\text{MinDeg}^+$, i.e., in the multigraph $H$, he already occupies the copies of these edges in $G$. That is, whatever non-Breaker edge Maker claims according to (3) in $G$ at the fixed vertex $v$, its copy in the graph $H$ cannot belong to Breaker at that point. Hence, the update in (4) is legal as well.

For (b), note that the edge claimed and directed in (4)
is an edge that Maker is allowed to choose in 
$\text{MinDeg}^+$ when following the strategy $\mathcal{S}^+$. Therefore, we can apply Lemma~\ref{lem:degree.game.variant} and obtain that,
as long as $d_{M\cap H}(v_i)\leq \frac{\eps s'n}{b'}$ holds,
we have 
$$d_{B\cap H}(v_i)<(1-\eps)d_H(v_i)
=(1-\eps)s'(n-1) < \left(1-\frac{\eps}{2}\right)\frac{sn}{2}
$$
for large enough $n$.
\end{proof}

\begin{claim}\label{clm:conn.large.s:degree.bounds}
As long as Maker does not forfeit the game, we have the following in the actual game:
For every vertex $v$ and every $i\in [2]$, as long as 
$d^+_{M,S_i}(v) < \frac{\eps sn}{4b}$ we have
$d_{B,S_i}(v) < (1-\frac{\eps}{4})\frac{sn}{2}$.
\end{claim}

\begin{proof}
First let us observe the following: whenever a sequence of rounds $1+t\cdot i,\ldots,t(i+1)$
with $i\in \mathbb{N}_0$ is finished,
Breaker has the same edges in the actual game and the simulated game $\text{MinDeg}^+$, and Maker has the same directed edges in both games (as moves of type 2 do not use directions), if we identify $v_1=v_2$ for every $v\in V(G)$. 

Now, assume that 
$d^+_{M,S_i}(v) < \frac{\eps sn}{4b}$ holds
for any vertex $v$ and any $i\in [2]$
in any round $r$. Let $r'<r$ be the largest integer
such that $r'\equiv 0$ (mod $t$). 
At the end of the $r'$-th round of the actual game, 
we have that in the simulated game 
$\text{MinDeg}^+$, the vertex $v_i$ is active,
since 
$$d^+_{M\cap H}(v_i) = d^+_{M,S_i}(v) < \frac{\eps sn}{4b} <
\frac{\eps s'n}{b'} .
$$
Hence, at that point
$d_{B\cap H}(v_i) < (1-\frac{\eps}{2})\frac{sn}{2}$ by Claim~\ref{lem:MinDeg.legal}.
By the above observation we conclude that
$d_{B,S_i}(v) < (1-\frac{\eps}{2})\frac{sn}{2}$
at the end of round $r'$.
Until the end of round $r$, Breaker can claim at most
$(r-r')b \leq tb$ additional edges, leading to
$$d_{B,S_i}(v) < \left(1-\frac{\eps}{2}\right)\frac{sn}{2} + tb 
< \left(1-\frac{\eps}{4}\right)\frac{sn}{2}
$$
at the end of round $r$, provided $n$ is large enough
and since $tb=o(sn)$.
\end{proof}

\begin{claim}\label{clm:conn.large.s:type.1.works}
Maker can always follow her strategy for the moves
of type 1. At the end, for every vertex $v$ and every $i\in [2]$, we have $d_{M,S_i}(v) = \frac{\eps s'n}{b'}$. Moreover, whenever Maker claims a random edge at a vertex $v$ in color $c\in S_i$ according to part (3) of the strategy, she chooses u.a.r.~from a set of at least $\frac{\eps n}{5}$ edges. 
\end{claim}

\begin{proof}
Assume that $d_{M,S_i}(v) < \frac{\eps s'n}{b'}$ (and hence $v_i$ is active) and that, according to a move of type 1, Maker has to claim an edge incident with $v$.
We have $d_{B,S_i}(v) < (1-\frac{\eps}{4})\frac{sn}{2}$ by the previous claim,
and $d_{M,S_i}(v) = o(|S_i|)$, since $s \gg \log(n)$. Hence, there must be a colour $c\in S_i$
such that Maker has no outgoing edge at $v$ of color $c$ yet and additionally
$$
d_{B,c}(v) \leq \frac{d_{B,S_i}(v)}{(1-o(1))|S_i|} < \left(1-\frac{\eps}{5}\right) n .
$$

According to (3), Maker chooses such a color $c$ for her strategy,
and then claims an edge u.a.r.~from the at least $\frac{\eps n}{5}$ non-Breaker edges at $v$ in color $c$.
The vertex $v$ will be chosen in moves of type 1 
 until $d_{M,S_i}(v) = \frac{\eps s'n}{b'}$ holds.
Once this value is reached, and hence $v_i$ is not active anymore, Maker does not increase this value any further, due to strategy $\mathcal{S}^+$.
\end{proof}

Having proven the above claim, we can now conclude a
suitable expansion property that is likely to hold in Maker's directed graph.

\begin{claim}\label{clm:conn.large.s:expansion}
A.a.s.~the following holds: 
Let $A\subset V(G)$ with $a:=|A|\leq \frac{n}{\log^{(2)}(n)}$,
let $T$ be any rooted tree on at most $2a$ vertices
in $V(G)$, with depth at most $\log(n)/\log^{(5)}(n)$
and $A$ being the leaf set of $T$.
Let $i\in [2]$,
let $f:E(T)\rightarrow S_i$ be an edge-coloring of $T$,
and for each $v\in A$, let
$S_v$ be the set of colors appearing on the unique path in $T$ between the root and $v$.
Let 
$$
N^+_{T,f}(A) :=
\{w\in V(G)\setminus V(T): \exists ~ v\in A:~ 
w\in N_{M,c}^+(v)\text{~ for some ~} c\in S_i\setminus S_v\}
$$
Then
$$
|N^+_{T,f}(A)| > \log^{(4)}(n) \cdot |A| .
$$
\end{claim}

\begin{proof}
If the statement fails,
there must be a choice for 
$A,T,i,f$ as described above
and a set $B$ of size
$|B| = \log^{(4)}(n)|A|$ such that 
$N^+_{T,f}(A)\subseteq B$.
Our goal is to show that this happens 
with probability $o(1)$.

For the moment, let $A,T,i,f,B$ be fixed as described.
At each vertex in $A$, Maker claims 
$\frac{\eps s'n}{b'} > \frac{\eps}{3\gamma} \log(n)$ 
outgoing edges with distinct colors in $S_i$, according to (3). Among these colors,
only $|S_v|=o(\log(n))$ belong to $S_v$,
due to the depth of the tree. For each of the other 
$(1-o(1)) \frac{\eps}{3\gamma} \log(n) > \frac{\eps}{2} \log(n)$ colors, 
when Maker claims an outgoing edge at $v$ of this color, the probability of this edge ending in $V(T)\cup B$ is at most
$$
\frac{|V(T)|+|B|}{\eps n / 5} <
\frac{10|A|\log^{(4)}(n)}{\eps n},
$$
because of Claim~\ref{clm:conn.large.s:type.1.works}.
Hence, for $A,T,i,f,B$ fixed as above, 
the probability of having 
$N^+_{T,f}(A)\subseteq B$ is below
$$
\left( \frac{10|A|\log^{(4)}(n)}{\eps n} \right)^{0.5\eps \log(n) |A|} .
$$
Doing a union bound over all possible ways to choose $A,T,i,f,B$,
we see that the desired property fails with probability at most 
\begin{align*}
\sum_{a=1}^{n/\log^{(2)}(n)}
\binom{n}{a} 				
\cdot n^{4a} 			    
\cdot 2						
\cdot \left( \frac{s}{2} \right)^{2a} 	
\cdot \binom{n}{a \log^{(4)}(n)} 	    
\cdot 
\left( \frac{10a\log^{(4)}(n)}{\eps n} \right)^{0.5\eps \log(n) a} ,
\end{align*}
where we use that the number of sets $A$ and $B$ of sizes $a$ and $a \log^{(4)}(n)$ can be bounded with $\binom{n}{a}$ and $\binom{n}{a \log^{(4)}(n)}$ respectively;
the number of rooted trees of size at most $2a$ (in a vertex set of size $n$) can be bounded, using Cayley's formula,
by $\sum_{t=1}^{2a} \binom{n}{t} t^{t-1} < \sum_{t=1}^{2a} n^{2t-1} < n^{4a}$;
and the number of functions $f:E(T)\rightarrow S_i$ equals $|S_i|^{|E(T)|}< \left(\frac{s}{2}\right)^{|E(T)|}$.

This bound on the failure probability can be bounded further as follows:
\begin{align*}
& \sum_{a=1}^{n/\log^{(2)}(n)}
n^{7a} 
\cdot n^{a \log^{(4)}(n)} 	
\cdot 
\left( \frac{10a\log^{(4)}(n)}{\eps n} \right)^{0.5 \eps \log(n) a} \\
\leq &
\sum_{a=1}^{n/\log^{(2)}(n)}
\exp\left( 2a \log^{(4)}(n) \log(n)
+ 0.5 \eps \log(n) a \Big[ \log(a) + \log^{(5)}(n) - \log(n) + O(1) \Big] \right) \\
\leq &
\sum_{a=1}^{n/\log^{(2)}(n)}
\exp\left( 2 a \log^{(4)}(n) \log(n)
+ 0.5\eps \log(n) a \Big[ - \log^{(3)}(n) + \log^{(5)}(n) + O(1) \Big] \right) \\
\leq &
\sum_{a=1}^{n/\log^{(2)}(n)}
\exp\left( - 0.1 \eps \log(n) a \log^{(3)}(n) \right) = o(1) .
\end{align*}
This proves the claim.
\end{proof}

Let us now turn to the moves of type 2. We let
$$
\mathcal{F}
=
\left\{\cup_{c\in S_3} E_c(A,B):~ A\cap B=\varnothing,~
|A|=|B|=\frac{n}{\log^{(2)}(n)}\right\}.
$$
We claim that, by playing only every $t$ rounds a move of type 2, Maker can get an edge in each $F\in \mathcal{F}$.
To do so, we apply Theorem~\ref{thm:beck_criterion}
where Maker takes over the role of $\mathcal{F}$-Breaker,
and where $\mathcal{F}$-Maker has bias $tb$. It holds that
\begin{align*}
\sum_{F\in \mathcal{F}} 2^{-|F|/(tb)}
& \leq
\binom{n}{n/\log^{(2)}(n)}^2 \cdot 2^{-\left( \frac{n}{\log^{(2)}(n)}\right)^2\cdot \frac{s}{\log^{(2)}(n)} / tb} 
\leq 
2^{2n} \cdot 2^{- \frac{n \log(n)}{\gamma (\log^{(2)}(n))^4}} 
= o(1) .
\end{align*}
Hence, by Theorem~\ref{thm:beck_criterion}, we know that
Maker (in the role of Breaker) can claim an element in each $F\in \mathcal{F}$. In particular, she can follow her strategy completely.

\medskip

From now on, let us condition on the likely event of
Claim~\ref{clm:conn.large.s:expansion}.
It remains to explain how we can find rainbow paths
between every pair of vertices $x,y\in V(G)$.

\smallskip

At first we can find a tree $T_x$ with root $x$ such that all root-leaf paths are rainbow in $S_1$ and the number of leaves is at least $\frac{n}{\log^{(2)}(n)}$. 
We do this recursively as follows:
\begin{itemize}
\item Let $T_0$ consist of $x$ only.
\item For $i\leq \frac{\log(n)}{\log^{(5)}(n)}$, 
obtain a new tree $T_i$ from $T_{i-1}$ as follows.
Let $A_{i-1}$ be the leaves of $T_{i-1}$. 
If $|A_{i-1}|\geq \frac{n}{\log^{(2)}(n)}$, let $T_i:=T_{i-1}$.
Otherwise, let the colors in $T_{i-1}$ induce a function $f:E(T_{i-1}) \rightarrow S_1$,
and 
for each vertex $v\in A_{i-1}$, let $S_v$ be the set of colors on the unique $x,v$-path in $T_{i-1}$.
Then set $A_i := N^+_{T,f}(A_{i-1})$ and note that by
Claim~\ref{clm:conn.large.s:expansion}, we get
$|A_i| > \log^{(4)}(n) \cdot |A_{i-1}|$.
Then to the tree $T_{i-1}$
add all the vertices $A_i$, and for each $w\in A_i$ 
add exactly one directed Maker edge $(v,w)$ with $v\in A_{i-1}$ which has a color $c\in S_1\setminus S_v$.
Call the resulting tree $T_i$.
\item Finally, let $T_x := T_{\frac{\log(n)}{\log^{(5)}(n)}}$.
\end{itemize}
By Claim~\ref{clm:conn.large.s:expansion}, we have
$|A_i|\geq (\log^{(4)}(n))^i$ as long as 
$T_i\neq T_{i-1}$, and a quick calculation shows that we reach size $|A_i|\geq \frac{n}{\log^{(2)}(n)}$ before $i = \frac{\log(n)}{\log^{(5)}(n)}$.

\smallskip

Now, analogously we can find a tree $T_y$ with root $y$ such that all root-leaf paths are rainbow in $S_2$ and the number of leaves is at least $\frac{n}{\log^{(2)}(n)}$.
If the trees $T_x$ and $T_y$ intersect in a vertex, we immediately find a rainbow path between $x$ and $y$. Otherwise, by the moves of type 2, we can find an edge with a color from $S_3$ that connects leaves from $T_x$ and $T_y$, hence again leading to a rainbow path between $x$ and $y$.
\end{proof}

\section{Rainbow spanning tree game}
\label{sec:rainbow-tree}

\begin{proof}[Proof of Theorem~\ref{thm:rainbow-spanning}]
We start with the upper bound 
on the threshold bias. Let $b\geq \frac{1}{\log(n-1)}\binom{n}{2}+1$, then Breaker can win as follows. 
Consider the game Box$(b;a_1,\ldots,a_{n-1})$
where each box $F_i$ has size $a_i=\binom{n}{2}$ and contains all edges of the $i$-th layer $G_i$. By Theorem~\ref{thm:boxgame},
Breaker (in the role of Maker) can ensure to occupy a box completely. That is, Maker does not get any edge from this layer and therefore cannot create a rainbow spanning tree.

\bigskip

Hence, it remains to prove the lower bound.
By bias monotonicity we may assume 
$b = \frac{\alpha n^2}{\log(n)}$ with 
a constant $\alpha < \frac{\log(2)}{8}$. 
Identify each layer $G_i$ with a unique color $c_i$, fix $\mathcal{C}:=\{c_1,\ldots,c_{n-1}\}$ as the given color set, and write $G_{c_i}:=G_i$. Moreover, let $\mathcal{P}$ be the set of all unordered partitions (without empty parts) of the vertex set $V$.
By Theorem~\ref{thm:rainbow.tree.equiv} we know that, if Maker fails to win the game,
then by the end of the game there must exist a partition $\{V_1,\ldots,V_k\}\in \mathcal{P}$ such that Maker claims
edges of at most $k-2$ colors between the parts $V_i$, i.e., that have endpoints in different parts of the partition. In order to avoid this to happen, it is enough for Maker to
ensure that for every $k\in \{2,\ldots,n\}$, every partition $\{V_1,\ldots,V_k\}\in \mathcal{P}$ and every set $S\subset \mathcal{C}$ of $n-k+1$ colors,
Maker manages to claim at least one edge
that has a color from $S$ and lies between
two different parts of the partition.

To make this more precise, let us use the following notation:
For any partition $\{V_1,\ldots,V_k\}\in \mathcal{P}$ 
and color $c\in \mathcal{C}$,
let us denote with
$$
E_c(V_1,\ldots,V_k)
=
\{vw\in E(G_c): 
~(\exists i\neq j:~v\in V_i,~ w\in V_j)\}
$$
the set of all edges of color $c$ between the partition,
and let $e(V_1,\ldots,V_k)=|E_c(V_1,\ldots,V_k)|$.
By what we just said, we know that Maker wins the game, if she can get an element in each set of the following family:
$$
\mathcal{F}:=
\left\{
\bigcup_{c\in S} E_c(V_1,\ldots,V_k):~
2\leq k\leq n,~ \{V_1,\ldots,V_k\}\in\mathcal{P},~
S\subset \mathcal{C},~
|S|=n-k+1
\right\} .
$$
By Beck's generalization of the Erd\H{o}s-Selfridge Criterion (Theorem~\ref{thm:beck_criterion}),
and switching the roles of Maker and Breaker,
it is enough to prove that
\begin{equation}\label{eq:Beck:spanning.tree}
\sum_{F\in \mathcal{F}} 2^{-|F|/b}<1    
\end{equation}
holds. In order to show this, we distinguish three kinds of partitions in $\mathcal{P}$:
\begin{align*}
\mathcal{P}_1 & := \Big\{ \{V_1,\ldots,V_k\}\in \mathcal{P}:~ 2\leq k\leq \frac{n}{2},~ \text{and}~ |V_i|\leq \frac{n}{2}~ \text{for every }i\in [k]\Big\}   ,\\
\mathcal{P}_2 & := \Big\{ \{V_1,\ldots,V_k\}\in \mathcal{P}:~ 2\leq k\leq n,~ \text{and}~ |V_i| > \frac{n}{2}~ \text{for some }i\in [k]\Big\} , \\
\mathcal{P}_3 & := \Big\{ \{V_1,\ldots,V_k\}\in \mathcal{P}:~ \frac{n}{2} < k\leq n \Big\} ,
\end{align*}
and for $i\in [3]$ we set
$$
\mathcal{F}_i:=
\left\{
\bigcup_{c\in S} E_s(V_1,\ldots,V_k):~
2\leq k\leq n,~ \{V_1,\ldots,V_k\}\in\mathcal{P}_i,~
S\subset \mathcal{C},~
|S|=n-k+1
\right\} 
$$
so that $\mathcal{F}=\mathcal{F}_1\cup \mathcal{F}_2\cup \mathcal{F}_3$.   

\bigskip

\textbf{Calculation for $\mathcal{F}_1$.}
Let $\{V_1,\ldots,V_k\}\in \mathcal{P}_1$.
Then 
\begin{align*}
e(V_1,\ldots,V_k) 
= \frac{1}{2} \sum_{i\in [k]} e(V_i,V\setminus V_i)
& = \frac{1}{2} \sum_{i\in [k]} |V_i|(n-|V_i|) \\
& = \frac{1}{2}n^2 - \frac{1}{2}\sum_{i\in [k]} |V_i|^2
\geq 
\frac{1}{2}n^2 - \frac{1}{4}n^2
=
\frac{1}{4}n^2 ,
\end{align*}
where the inequality holds,
since a sum $\sum_{i\in k} a_i^2$ with 
constraints $a_i\in [0,\frac{n}{2}]$
and $\sum_{i\in [k]}a_i = n$
gets largest, if two summands $a_i$ equal $\frac{n}{2}$, while the others are $0$.
Therefore, we can estimate as follows:
\begin{align*}
\sum_{F\in \mathcal{F}_1} 2^{-|F|/b} 
& =
\sum_{k=2}^{n/2}~
\sum_{\{V_1,\ldots,V_k\}\in \mathcal{P}_1}~
\sum_{S\in \binom{\mathcal{C}}{n-k+1}}
2^{-|S|\cdot e(V_1,\ldots,V_k)/b} \\
& \leq 
\sum_{k=2}^{n/2}~
\sum_{\{V_1,\ldots,V_k\}\in \mathcal{P}_1}~ \binom{n-1}{n-k+1}
2^{-(n-k+1)\cdot \frac{n^2}{4}/b} 
 \leq 
n^n 2^n 
2^{-\frac{1}{8\alpha} n\log(n)} = o(1)
\end{align*}
where the second inequality uses the choice of $b$ as well as the simple bounds
$|\mathcal{P}_1|\leq |\mathcal{P}|\leq n^n$,
$\binom{n-1}{n-k+1}\leq 2^n$
and $n-k+1\geq \frac{n}{2}$,
and the last estimate uses the choice of $\alpha$.

\medskip

\textbf{Calculation for $\mathcal{F}_2$.}
Let $\{V_1,\ldots,V_k\}\in \mathcal{P}_2$.
We may assume that $|V_1|=n-t$ with $t<\frac{n}{2}$. We then have
$$e(V_1,\ldots,V_k) \geq
e(V_1,V\setminus V_1) = t(n-t)$$
and, since $k\leq t+1$ must hold,
we obtain $n-k+1 \geq n-t > \frac{n}{2}$,
which leads to
$$
\binom{n-1}{n-k+1}
<
\binom{n}{n-k+1}
\leq
\binom{n}{n-t}
=
\binom{n}{t} .
$$
Moreover, we can upper bound the number of all partitions in $\mathcal{P}$ 
with $|V_1|=n-t$
by $\binom{n}{t} t^t$,
as there are $\binom{n}{t}$ options to choose $V_1$ and at most $t^t$ options to partition the remaining $t$ elements.
We can therefore estimate as follows:
\begin{align*}
\sum_{F\in \mathcal{F}_2} 2^{-|F|/b} 
& =
\sum_{k=2}^{n}~
\sum_{t=1}^{n/2}~
\sum_{\small \begin{array}{l}\{V_1,\ldots,V_k\}\in \mathcal{P}_2 \\ \text{s.t.~}|V_1|=n-t \end{array}}~
\sum_{S\in \binom{\mathcal{C}}{n-k+1}}
2^{-|S|\cdot e(V_1,\ldots,V_k)/b} \\
& \leq 
\sum_{t=1}^{n/2}~
\sum_{k=2}^{n}~
\sum_{\small \begin{array}{l}\{V_1,\ldots,V_k\}\in \mathcal{P}_2 \\ \text{s.t.~}|V_1|=n-t \end{array}}~
\binom{n-1}{n-k+1}
2^{-(n-k+1)\cdot e(V_1,\ldots,V_k)/b} \\
& \leq 
\sum_{t=1}^{n/2}~
\binom{n}{t}t^t~
\binom{n}{t}
2^{-t(n-t)^2/b} \\
& \leq 
\sum_{t=1}^{n/\log(n)}~
\binom{n}{t}^2 t^t~
2^{-t(n-t)^2/b}
+
\sum_{t=n/\log(n)}^{n/2}~
\binom{n}{t}^2 t^t~
2^{-t(n-t)^2/b} 
\end{align*}
where the first sum can be estimated
with 
\begin{align*}
    \sum_{t=1}^{n/\log(n)}~
\binom{n}{t}^2 t^t~
2^{-t(n-t)^2/b}
& \leq \sum_{t=1}^{n/\log(n)}~
n^{3t}
2^{-(1-o(1))tn^2/b} \\
& 
\leq \sum_{t=1}^{n/\log(n)}
\exp\left( 3t\log(n) - (1-o(1))\frac{\log(2)}{\alpha} t\log(n) \right) = o(1)
\end{align*}
and the second sum can be estimated with
\begin{align*}
    \sum_{t=n/\log(n)}^{n/2}~
\binom{n}{t}^2 t^t~
2^{-t(n-t)^2/b}
& \leq \sum_{t=n/\log(n)}^{n/2}~
\left(\frac{en}{t}\right)^{2t} t^t~
2^{-tn^2/4b} \\
& \leq \sum_{t=n/\log(n)}^{n/2}~
\left(e \log(n)\right)^{2t} n^t~
2^{-tn^2/4b} \\
& \leq \sum_{t=n/\log(n)}^{n/2}~
\exp\left( (1+o(1))t\log(n) - \frac{\log(2)}{4\alpha} t\log(n) \right) = o(1)
\end{align*}
by the choice of $\alpha$.

\medskip

\textbf{Calculation for $\mathcal{F}_3$.}
For fixed $r<\frac{n}{2}$, we first want to estimate the number of partitions $\{V_1,\ldots,V_k\}\in \mathcal{P}$ with $k=n-r$. For this, observe that each such partition must have at least $n-2r$ parts of size 1. Indeed, if $x_1$ is the number of parts of size $1$ and $x_2$ is the number of parts of size at least $2$, then
$k=x_1+x_2$ and $n\geq x_1 + 2x_2$, and hence
$$
x_1 = k-x_2 = n-r - x_2 \geq x_1+x_2 - r = k - r = n-2r .
$$
It follows that the number of partitions 
with $k=n-r$ and $r<\frac{n}{2}$ can be bounded by  $\binom{n}{2r}\cdot (2r)^{2r}$, since there are $\binom{n}{2r}$ options to choose $n-2r$ parts of size $1$
and at most $(2r)^{2r}$ options to partition the rest, and this number can be further bounded as follows: 
$$\binom{n}{2r}\cdot (2r)^{2r} \leq 
\left( \frac{en}{2r} \right)^{2r} \cdot (2r)^{2r} =
(en)^{2r} = n^{(2+o(1))r} = n^{(2+o(1))(n-k)}.$$

Now let $\{V_1,\ldots,V_k\}\in \mathcal{P}_3$.
Then there exists $r<\frac{n}{2}$ such that $k=n-r$. The size of $e(V_1,\ldots,V_k)$
is minimal if all but one part have size $1$,
and the remaining part has size
$n-k+1=r+1$. This leads to 
$$
e(V_1,\ldots,V_k) \geq
\binom{n}{2} - \binom{r+1}{2}
\geq 
\left( \frac{3}{8} - o(1) \right)n^2
$$
Thus, we can do the following estimate:
\begin{align*}
\sum_{F\in \mathcal{F}_3} 2^{-|F|/b} 
& \leq 
\sum_{k=n/2}^{n}~
\sum_{\{V_1,\ldots,V_k\}\in \mathcal{P}_3}~
\sum_{S\in \binom{\mathcal{C}}{n-k+1}}
2^{-|S|\cdot e(V_1,\ldots,V_k)/b} \\
& \leq 
\sum_{k=n/2}^{n}~
n^{(2+o(1))(n-k)}~
n^{n-k+1}
2^{-(n-k+1)\cdot ( \frac{3}{8} - o(1))n^2/b} \\
& \leq 
\sum_{k=n/2}^{n}~
\exp\Big( (3+o(1))(n-k+1)\log(n) - \frac{\log(2)}{\alpha}\left ( \frac{3}{8} - o(1)\right)(n-k+1)\log(n) \Big) \\
& = o(1) .
\end{align*}

\smallskip

Now, putting everything together we see that
\eqref{eq:Beck:spanning.tree} holds
for large enough $n$. This proves the theorem.
\end{proof}

\section{Concluding remarks}
\label{sec:concluding}

\subsection{Optimality of our results}

For the rainbow-connectivity game
$\mathcal{C}_{s,n}$,
we determine the order of the threshold bias,
provided that $s$ is a constant (Theorem~\ref{thm:rainbow-conn.s.constant}),
or $s=\omega(\log(n))$ holds
(Theorem~\ref{thm:rainbow-conn.s.large}).
Moreover, all calculations in the proof 
of Theorem~\ref{thm:rainbow-conn.s.constant}
go through as long as
$s=o(\sqrt{\log(n)})$,
and the proof of Lemma~\ref{clm:rainbow-conn.neighboroods} is the only place that actually requires this bound on $s$. We believe that a more thorough analysis of our strategy could improve this bound on $s$. Nevertheless, if $\frac{\log(n)}{\log^{(2)}(n)} \ll s \ll \log(n)$, new arguments will be needed, as 
many estimates in the proof of Theorem~\ref{thm:rainbow-conn.s.constant}
do not hold in this range. Accordingly we pose the following problem.

\begin{problem}
Determine the threshold bias of
the rainbow-connectivity game 
$\mathcal{C}_{s,n}$ 
when $s$ is in the range
$\sqrt{\log(n)} \ll s \ll \log(n)$.
\end{problem}

In the case when the number of colors satisfies
$s=\omega(\log(n))$, we even determine
the threshold bias of $\mathcal{C}_{s,n}$ up to a multiplicative factor $2$. We conjecture the following.

\begin{conjecture}
Let $s = \omega(\log(n))$. 
Then $b_{\mathcal{C}_{s,n}} = (1+o(1))\frac{sn}{\log(n)}$. 
\end{conjecture}

Similarly, we believe that our upper bound
in Theorem~\ref{thm:rainbow-spanning} for the rainbow-spanning tree game is asymptotically tight. 

\begin{conjecture}
We have $b_{\mathcal{RS}_{n}} = (1+o(1))\frac{n^2}{\log(n)}$.
\end{conjecture}

Finally, we note that the number of paths stated in Theorem~\ref{thm:rainbow-conn.s.constant}(b) is optimal up to the constant factor. This can be seen as follows:
Let $xy$ be the first edge that Maker
claims, and let $v,w\in V\setminus \{x,y\}$. Then define
$\mathcal{F}=\mathcal{F}_1\cup \mathcal{F}_2$ with
\begin{align*}
\mathcal{F}_1 & :=
\left\{
E(P):~P~\text{ is a $v,w$-path of length $s$ in $G$ that does not contain the edge $xy$}
\right\} , \\
\mathcal{F}_2 & :=
\left\{
E(P)\setminus\{xy\}:~P~\text{ is a $v,w$-path of length $s$ in $G$ that contains the edge $xy$}
\right\} .
\end{align*}
We can assume that Breaker from now on plays as first player. From the proof of Beck's generalization of the Erd\H{o}s-Selfridge Criterion~\cite{beck1982remarks} it follows that Breaker can ensure 
that Maker occupies at most
\begin{align*}
\sum_{F\in \mathcal{F}} (1+b)^{-|F|}
& = 
\sum_{F\in \mathcal{F}_1} (1+b)^{-|F|}
+
\sum_{F\in \mathcal{F}_2} (1+b)^{-|F|} \\
& =
O(n^{s-1} b^{-s}) +
O(n^{s-3} b^{-(s-1)}) =
O(n^{s-1} b^{-s}).
\end{align*}
winning sets of $\mathcal{F}$, i.e., $v,w$-paths of length $s$.

\subsection{More rainbow structures}

After considering rainbow spanning trees,
a natural next step is to look at games in which Maker's goal is to claim a 
perfect matching, or a Hamilton cycle, that is rainbow. 

\begin{definition}[Rainbow perfect matching game $\mathcal{RP}_{n}$] 
The $(1:b)$ game $\mathcal{RP}_{n}$ is played as follows.
Given copies $G_1,\ldots,G_{\frac{n}{2}}$ of the complete graph
$K_n$ on vertex set $V=[n]$, where $n$ is even, Maker and Breaker alternatingly claim edges of  the board $X=\bigcup_{i\in [\frac{n}{2}]}E(G_i)$ according to the given bias. Maker wins iff she manages to claim
a rainbow perfect matching.
\end{definition}

\begin{definition}[Rainbow Hamiltonicity game $\mathcal{RH}_{n}$] 
The $(1:b)$ game $\mathcal{RH}_{n}$ is played as follows.
Given copies $G_1,\ldots,G_{n}$ of the complete graph
$K_n$ on vertex set $V=[n]$, where $n$ is even, Maker and Breaker alternatingly claim edges of  the board $X=\bigcup_{i\in [n]}E(G_i)$ according to the given bias. Maker wins iff she manages to claim
a rainbow Hamilton cycle.
\end{definition}

We conjecture the following.

\begin{conjecture}
We have $b_{\mathcal{RP}_{n}} = (\frac{1}{2}+o(1))\frac{n^2}{\log(n)}$.
\end{conjecture}

\begin{conjecture}
We have $b_{\mathcal{RH}_{n}} = (1+o(1))\frac{n^2}{\log(n)}$.    
\end{conjecture}

\medskip

\section*{Acknowledgment}
The research of Juri Barkey, Bruno Borchardt and Dennis Clemens is supported financially by the Federal Ministry of Research, Technology and Space through the DAAD project 57749672.
Milica Maksimovi\'{c}, Mirjana Mikala\v{c}ki and Milo\v{s} Stojakovi\'{c} are supported by the Ministry of Science, Technological Development and Innovation of the Republic of Serbia through bilateral Serbia-Germany cooperation 001545243 2025 13440 003 000 000 001 03 005, and through grants 451-03-33/2026-03/200125 \& 451-03-34/2026-03/200125. Milo\v{s} Stojakovi\'{c} is supported by the Science Fund of the Republic of Serbia, Grant \#7462: Graphs in Space and Time: Graph Embeddings for Machine Learning in Complex Dynamical Systems (TIGRA).

\begin{center}
    \includegraphics[height=4cm]{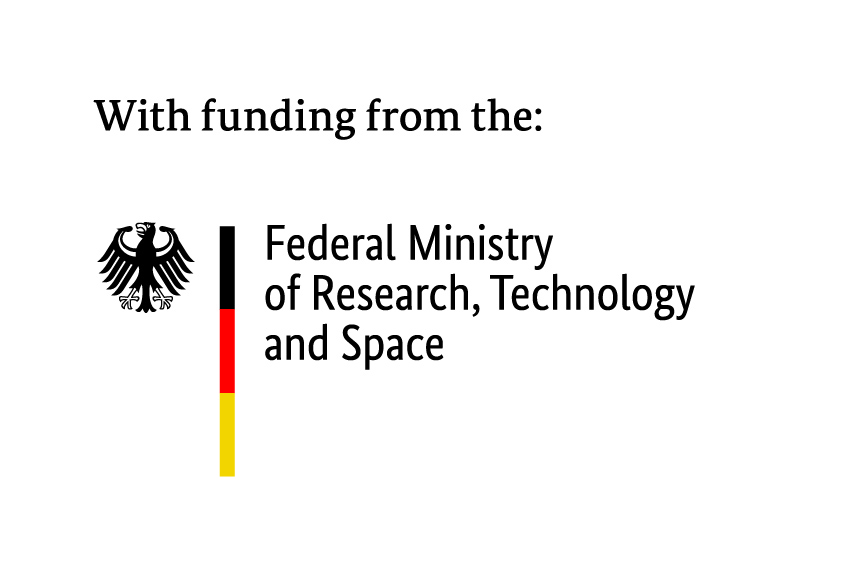}
    \hspace{0.5cm}
    \includegraphics[height=3cm]{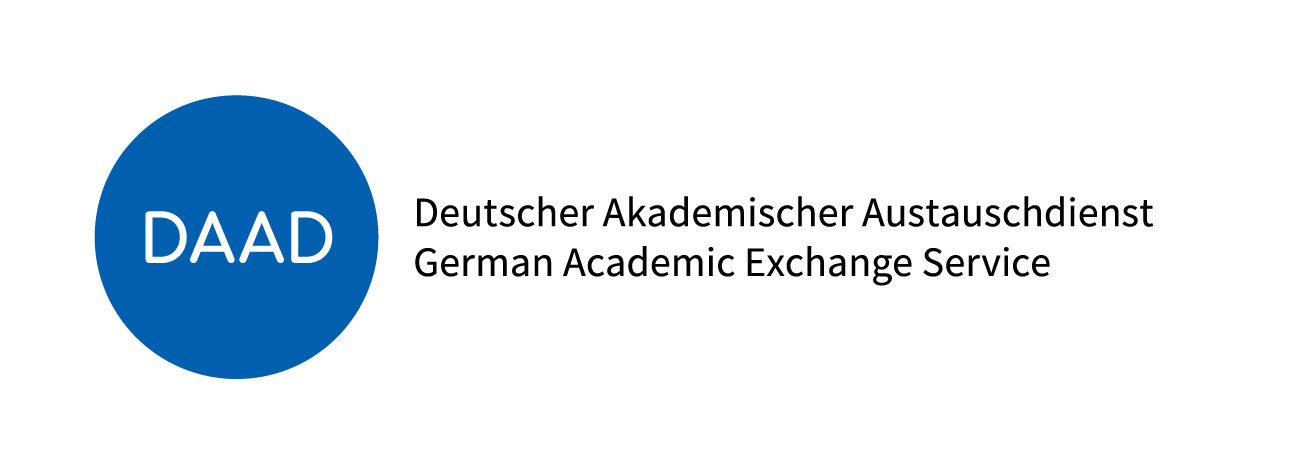}
\end{center}

\bibliographystyle{amsplain}
\bibliography{references}

@article{erdos1973combinatorial,
  title={On a combinatorial game},
  author={Erd{\H{o}}s, Paul and Selfridge, John L.},
  journal={Journal of Combinatorial Theory, Series A},
  volume={14},
  number={3},
  pages={298--301},
  year={1973},
  publisher={Elsevier}
}

@article{aharoni2020rainbow,
  title={A rainbow version of {M}antel’s {T}heorem},
  author={Aharoni, Ron and DeVos, Matt and de la Maza, Sebasti{\'a}n G.~H. and Montejano, Amanda and {\v{S}}{\'a}mal, Robert},
  journal={Advances in Combinatorics},
  year={2020}
}

@article{allen2017making,
  title={Making spanning graphs},
  author={Allen, Peter and B{\"o}ttcher, Julia and Kohayakawa, Yoshiharu and Naves, Humberto and Person, Yury},
  journal={arXiv preprint arXiv:1711.05311},
  year={2017}
}

@book{alon2016probabilistic,
  title={The probabilistic method},
  author={Alon, Noga and Spencer, Joel H.},
  year={2016},
  publisher={John Wiley \& Sons}
}

@article{anastos2023robust,
  title={Robust {H}amiltonicity in families of {D}irac graphs},
  author={Anastos, Michael and Chakraborti, Debsoumya},
  journal={arXiv preprint arXiv:2309.12607},
  year={2023}
}

@article{balogh2009diameter,
  title={The diameter game},
  author={Balogh, J{\'o}zsef and Martin, Ryan and Pluh{\'a}r, Andr{\'a}s},
  journal={Random Structures \& Algorithms},
  volume={35},
  number={3},
  pages={369--389},
  year={2009},
  publisher={Wiley Online Library}
}

@article{beck1982remarks,
  title={Remarks on positional games. I},
  author={Beck, J{\'o}zsef},
  journal={Acta Mathematica Hungarica},
  volume={40},
  number={1-2},
  pages={65--71},
  year={1982},
  publisher={Akad{\'e}miai Kiad{\'o}, co-published with Springer Science+ Business Media BV~…}
}

@article{bednarska2000biased,
  title={Biased positional games for which random strategies are nearly optimal},
  author={Bednarska, Ma{\l}gorzata and {\L}uczak, Tomasz},
  journal={Combinatorica},
  volume={20},
  number={4},
  pages={477--488},
  year={2000},
  publisher={Springer}
}

@article{bradshaw2021transversals,
  title={Transversals and {B}ipancyclicity in {B}ipartite {G}raph {F}amilies},
  author={Bradshaw, Peter},
  journal={The Electronic Journal of Combinatorics},
  pages={P4--25},
  year={2021}
}

@article{bradshaw2022one,
  title={From one to many rainbow {H}amiltonian cycles},
  author={Bradshaw, Peter and Halasz, Kevin and Stacho, Ladislav},
  journal={Graphs and Combinatorics},
  volume={38},
  number={6},
  pages={188},
  year={2022},
  publisher={Springer}
}

@inproceedings{bradshaw2021rainbow,
  title={A rainbow connectivity threshold for random graph families},
  author={Bradshaw, Peter and Mohar, Bojan},
  booktitle={Extended Abstracts EuroComb 2021: European Conference on Combinatorics, Graph Theory and Applications},
  pages={842--847},
  year={2021},
  organization={Springer}
}

@book{brualdi1991combinatorial,
	place={Cambridge}, 
	series={Encyclopedia of Mathematics and its Applications}, 
	title={Combinatorial Matrix Theory}, 
	publisher={Cambridge University Press}, 
	author={Brualdi, Richard A. and Ryser, Herbert J.}, 
	year={1991}, 
	collection={Encyclopedia of Mathematics and its Applications}
}

@article{chakraborti2023bandwidth,
  title={A bandwidth theorem for graph transversals},
  author={Chakraborti, Debsoumya and Im, Seonghyuk and Kim, Jaehoon and Liu, Hong},
  journal={arXiv preprint arXiv:2302.09637},
  year={2023}
}

@article{chartrand2008rainbow,
  title={Rainbow connection in graphs},
  author={Chartrand, Gary and Johns, Garry L. and McKeon, Kathleen A. and Zhang, Ping},
  journal={Mathematica bohemica},
  volume={133},
  number={1},
  pages={85--98},
  year={2008},
  publisher={Institute of Mathematics, Academy of Sciences of the Czech Republic}
}

@article{cheng2025transversal,
  title={Transversal {H}amilton {C}ycle in {H}ypergraph {S}ystems},
  author={Cheng, Yangyang and Han, Jie and Wang, Bin and Wang, Guanghui and Yang, Donglei},
  journal={SIAM Journal on Discrete Mathematics},
  volume={39},
  number={1},
  pages={55--74},
  year={2025},
  publisher={SIAM}
}

@article{cheng2021rainbow,
  title={Rainbow Pancyclicity in Graph Systems},
  author={Cheng, Yangyang and Wang, Guanghui and Zhao, Yi},
  journal={The Electronic Journal of Combinatorics},
  pages={P3--24},
  year={2021}
}

@incollection{chvatal1978biased,
  title={Biased positional games},
  author={Chv{\'a}tal, Va{\v{s}}ek and Erd{\H o}s, Paul},
  booktitle={Annals of Discrete Mathematics},
  volume={2},
  pages={221--229},
  year={1978},
  publisher={Elsevier}
}

@article{clemens2025maker,
  title={Maker playing against an invisible {B}reaker},
  author={Clemens, Dennis and Hamann, Fabian and Mikala{\v{c}}ki, Mirjana and Mogge, Yannick and Stojakovi{\'c}, Milo{\v{s}}},
  journal={arXiv preprint arXiv:2507.22519},
  year={2025}
}

@inproceedings{edmonds2003submodular,
  title={Submodular functions, matroids, and certain polyhedra},
  author={Edmonds, Jack},
  booktitle={Combinatorial Optimization—Eureka, You Shrink! Papers Dedicated to Jack Edmonds 5th International Workshop Aussois, France, March 5--9, 2001 Revised Papers},
  pages={11--26},
  year={2003},
  organization={Springer}
}

@article{ferber2022dirac,
  title={Dirac-type problem of rainbow matchings and {H}amilton cycles in random graphs},
  author={Ferber, Asaf and Han, Jie and Mao, Dingjia},
  journal={arXiv preprint arXiv:2211.05477},
  year={2022}
}

@article{ferber2015generating,
  title={Generating random graphs in biased {M}aker-{B}reaker games},
  author={Ferber, Asaf and Krivelevich, Michael and Naves, Humberto},
  journal={Random Structures \& Algorithms},
  volume={47},
  number={4},
  pages={615--634},
  year={2015},
  publisher={Wiley Online Library}
}

@book{frieze2015introduction,
  title={Introduction to random graphs},
  author={Frieze, Alan and Karo{\'n}ski, Micha{\l}},
  year={2015},
  publisher={Cambridge University Press}
}

@article{gebauer2009asymptotic,
  title={Asymptotic random graph intuition for the biased connectivity game},
  author={Gebauer, Heidi and Szab{\'o}, Tibor},
  journal={Random Structures \& Algorithms},
  volume={35},
  number={4},
  pages={431--443},
  year={2009},
  publisher={Wiley Online Library}
}

@article{gupta2023general,
  title={A general approach to transversal versions of {D}irac-type theorems},
  author={Gupta, Pranshu and Hamann, Fabian and M{\"u}yesser, Alp and Parczyk, Olaf and Sgueglia, Amedeo},
  journal={Bulletin of the London Mathematical Society},
  volume={55},
  number={6},
  pages={2817--2839},
  year={2023},
  publisher={Wiley Online Library}
}

@article{hales1963regularity,
title={Regularity and positional games},
author={Hales, Alfred W.  and Jewett, Robert I.},
journal={Transactions of the American Mathematical Society},
volume={106},
year={1963}, 
pages={222--229}
}

@article{heath2025universality,
  title={Universality for transversal powers of {H}amilton cycles},
  author={Heath, Emily and Hyde, Joseph and Morrison, Natasha and Ogden, Shannon},
  journal={arXiv preprint arXiv:2510.18163},
  year={2025}
}

@article{stojakovic2005positional,
  title={Positional games on random graphs},
  author={Stojakovi{\'c}, Milo{\v{s}} and Szab{\'o}, Tibor},
  journal={Random Structures \& Algorithms},
  volume={26},
  number={1-2},
  pages={204--223},
  year={2005},
  publisher={Wiley Online Library}
}

@article{nenadov2016threshold,
  title={On the threshold for the {M}aker-{B}reaker {$H$}-game},
  author={Nenadov, Rajko and Steger, Angelika and Stojakovi{\'c}, Milo{\v{s}}},
  journal={Random Structures \& Algorithms},
  volume={49},
  number={3},
  pages={558--578},
  year={2016},
  publisher={Wiley Online Library}
}

@article{hefetz2008planarity,
  title={Planarity, colorability, and minor games},
  author={Hefetz, Dan and Krivelevich, Michael and Stojakovi{\'c}, Milo{\v{s}} and Szab{\'o}, Tibor},
  journal={SIAM Journal on Discrete Mathematics},
  volume={22},
  number={1},
  pages={194--212},
  year={2008},
  publisher={SIAM}
}

@book{hefetz2014positional,
  title={Positional games},
  author={Hefetz, Dan and Krivelevich, Michael and Stojakovi{\'c}, Milo{\v{s}} and Szab{\'o}, Tibor},
  volume={44},
  year={2014},
  publisher={Springer}
}

@book{janson2011random,
	title={Random graphs},
	author={Janson, Svante and Ruci\'nski, Andrzej and {\L}uczak, Tomasz},
	year={2011},
	publisher={John Wiley \& Sons}
}

@article{joos2020rainbow,
  title={On a rainbow version of {D}irac's theorem},
  author={Joos, Felix and Kim, Jaehoon},
  journal={Bulletin of the London Mathematical Society},
  volume={52},
  number={3},
  pages={498--504},
  year={2020},
  publisher={Wiley Online Library}
}

@article{krivelevich2014positional,
  title={Positional games},
  author={Krivelevich, Michael},
  journal={arXiv preprint arXiv:1404.2731},
  year={2014}
}

@article{krivelevich2011critical,
  title={The critical bias for the {H}amiltonicity game is 
  {$(1+ o(1))n/\ln n$}},
  author={Krivelevich, Michael},
  journal={Journal of the American Mathematical Society},
  volume={24},
  number={1},
  pages={125--131},
  year={2011}
}

@article{liebenau2022threshold,
  title={The threshold bias of the clique-factor game},
  author={Liebenau, Anita and Nenadov, Rajko},
  journal={Journal of Combinatorial Theory, Series B},
  volume={152},
  pages={221--247},
  year={2022},
  publisher={Elsevier}
}

@article{montgomery2023proof,
  title={A proof of the {R}yser-{B}rualdi-{S}tein conjecture for large even $ n$},
  author={Montgomery, Richard},
  journal={arXiv preprint arXiv:2310.19779},
  year={2023}
}

@article{nenadov2023probabilistic,
  title={Probabilistic intuition holds for a class of small subgraph games},
  author={Nenadov, Rajko},
  journal={Proceedings of the American Mathematical Society},
  volume={151},
  number={04},
  pages={1495--1501},
  year={2023}
}

@article{ryser1967neuere,
  title={Neuere {P}robleme der {K}ombinatorik},
  author={Ryser, Herbert J.},
  journal={Vortr{\"a}ge {\"u}ber Kombinatorik, Oberwolfach},
  volume={69},
  number={91},
  pages={35},
  year={1967},
  publisher={Matematisches Forschungsinstitute Oberwolfach, Germany}
}

@book{schrijver2003combinatorial,
  title={Combinatorial optimization: polyhedra and efficiency},
  author={Schrijver, Alexander},
  volume={24},
  number={2},
  year={2003},
  publisher={Springer}
}

@article{stein1975transversals,
  title={Transversals of {L}atin squares and their generalizations},
  author={Stein, Sherman K.},
  journal={Pacific Journal of Mathematics},
  volume={61},
  number={2},
  pages={567--575},
  year={1975}
}

@article{sun2024transversal,
  title={Transversal structures in graph systems: A survey},
  author={Sun, Wanting and Wang, Guanghui and Wei, Lan},
  journal={arXiv preprint arXiv:2412.01121},
  year={2024}
}

\appendix
\section{Minimum degree game}\label{sec:min.deg}

\begin{proof}[Sketch of proof of Lemma~\ref{lem:degree.game.variant}]
The proof is almost the same as for Theorem 1.2 in~\cite{gebauer2009asymptotic}. However, we follow the structure and notation given in~\cite{hefetz2014positional}.

W.l.o.g.~let $\delta<0.5$. We set $\eps'=0.1\delta$,
and let $\eps\in (0,\eps')$.
Whenever necessary we assume that $n$ is large enough.
Suppose for contradiction that Breaker has a strategy
to ensure that at some point in the game, say in a round $g$,
there appears a vertex $v_g$ with
$d_M^+(v_g)<\frac{\eps s'n}{b}$ and
$d_B(v_g)\geq (1-\eps) d_H(v)$. 
Let $v_1,\ldots,v_{g-1}$ be the vertices chosen by Maker
in the previous rounds according to strategy~$\mathcal{S}^+$.
Let $J_i=\{v_{i+1},\ldots,v_g\}$ for every $0\leq i\leq g-1$.
Before Maker's move in round $i$, let
$$
\overline{\text{dang}}(M_i) = \frac{\sum_{v\in J_{i-1}} \text{dang}(v)}{|J_{i-1}|}
$$
and before Breaker's $i$-th move, let
$$
\overline{\text{dang}}(B_i) = \frac{\sum_{v\in J_{i}} \text{dang}(v)}{|J_{i}|} .
$$
By assumption, after Maker's $(g-1)$-th move,
we must have $d_B(v_g)\geq (1-\eps)d_H(v_g)-b$ and 
$d_M^+(v_g)< \frac{\eps s'n}{b}$, leading to
$$
\overline{\text{dang}}(M_g) 
\geq (1-\eps)s'(n-1) - b - 2b\cdot \frac{\eps s'n}{b}
> (1-4\eps)s'n .
$$

By following the proofs of Lemma~5.3.3, Lemma~5.3.4 and Corollary~5.3.5 in~\cite{hefetz2014positional} line by line, the following can be shown. 

\begin{claim} \label{clm:mindeg.analysis}
For every $0\leq i\leq g-1$, we have
\begin{enumerate}
\item[(a)] $\overline{\text{dang}}(M_i) \geq \overline{\text{dang}}(M_{i+1})$ if $J_i=J_{i-1}$,
\item[(b)] $\overline{\text{dang}}(M_i) \geq \overline{\text{dang}}(M_{i+1}) - \min\left\{ \frac{2b}{|J_i|} , \frac{b+a(i)-a(i-1)}{|J_i|} +  s' \right\}$
\end{enumerate}
where $a(i)$ is the number of Breaker edges in $J_i$ that are claimed in the first $i$ rounds.
\end{claim}

We skip the proof of this claim and note that
the only difference, when compared with
Corollary~5.3.5 in~\cite{hefetz2014positional}, is the summand $s'$ at the end of (b). This comes from the fact that now every vertex in $J_i$ is incident to less than 
$s'|J_i|$ edges within the set $J_i$.

Then, having Claim~\ref{clm:mindeg.analysis}, the most interesting rounds are those with $J_i \neq J_{i-1}$. Following the proof
after Corollary 5.3.5 in~\cite{hefetz2014positional}, we let
$1 \leq i_1 < \ldots < i_r \leq g-1$ be the indices of the rounds where this happens. If $r > k :=\frac{n}{\log(n)}$ we get
\begin{align*}
0 = \overline{\text{dang}}(M_1) \geq 
\overline{\text{dang}}(M_g) - \frac{b}{1} - \frac{b}{2} - \ldots - \frac{b}{k} - ks' - \frac{2b}{k+1} - \frac{2b}{k+1} - \ldots - \frac{2b}{r} 
\end{align*}
by the same argument as in~\cite{hefetz2014positional},
where the only difference to~\cite{hefetz2014positional} is the term "$-ks'$", due to change needed in Claim~\ref{clm:mindeg.analysis}.
Estimating further we get
\begin{align*}
0  = \overline{\text{dang}}(M_1)
& \geq 
(1-4\eps)s'n - b(\log(k) + 1) - ks' - 2b(\log(2n) - \log(k)) \\
& \geq 
(1-4\eps)s'n - b(\log(n) + 1) - 2b - 2b(\log(2)+\log\log(n)) \\
& \geq 
(1-4\eps)s'n - (1-\delta)s'n - o(s'n) 
=
(\delta - 4\eps - o(1))s'n > 0,
\end{align*}
a contradiction. 

If instead $r \leq  k :=\frac{n}{\log(n)}$ we get
\begin{align*}
0 = \overline{\text{dang}}(M_1) \geq 
\overline{\text{dang}}(M_g) - \frac{b}{1} - \frac{b}{2} - \ldots - \frac{b}{r} - rs' > 0
\end{align*}
again by the same argument as in~\cite{hefetz2014positional}, giving a contradiction.
\end{proof}

\end{document}